\documentclass[11pt]{amsart}

\usepackage{
  amsmath,
  amsfonts,
  amssymb,
  amsthm,
  amscd,
  graphicx,
  comment,      
  etoolbox,     
  mathtools,    
  enumitem,
  mathdots,     
  stmaryrd,     
  txfonts,      
  booktabs,     
  stackrel,     
}
\usepackage[usenames,dvipsnames]{xcolor}
\usepackage[all]{xy}

\usepackage[scaled=0.88]{beraserif} 
\usepackage[scaled=0.85]{berasans}  
\usepackage[scaled=0.84]{beramono}  
\usepackage[T1]{fontenc}
\usepackage{textcomp}               
\usepackage[sc]{mathpazo}           
\usepackage[T1,small,euler-digits,euler-hat-accent]{eulervm}  
\usepackage{mathrsfs}               

\usepackage[colorlinks=true, linkcolor=blue, citecolor=blue, urlcolor=blue, breaklinks=true]{hyperref}

\usepackage[capitalize]{cleveref}   

\usepackage{tikz}
\usetikzlibrary{arrows.meta}
\usetikzlibrary{decorations.markings}

\newcommand\redcircle[1]{\filldraw[fill=white, draw=red] #1 circle (2pt)}
\newcommand\bluedot[1]{\filldraw[blue] #1 circle (2pt)}

\newcommand\dotlabel[1]{$\scriptstyle{#1}$}

\newcommand\regionlabel[1]{$\scriptstyle{#1}$}

\newcommand\tokenlabel[1]{$\scriptstyle{#1}$}

\tikzset{anchorbase/.style={>=To,baseline={([yshift=-0.5ex]current bounding box.center)}}}
\tikzset{wipe/.style={white,line width=4pt}}

\newcommand\rcup{
  \begin{tikzpicture}[anchorbase]
    \draw[->] (0,0.1) to (0,0) arc(180:360:0.2) to (0.4,0.1);
  \end{tikzpicture}
}

\newcommand\lcup{
  \begin{tikzpicture}[anchorbase]
    \draw[<-] (0,0.1) to (0,0) arc(180:360:0.2) to (0.4,0.1);
  \end{tikzpicture}
}

\newcommand\rcap{
  \begin{tikzpicture}[anchorbase]
    \draw[->] (0,-0.1) to (0,0) arc(180:0:0.2) to (0.4,-0.1);
  \end{tikzpicture}
}

\tikzset{->-/.style={decoration={
  markings,
  mark=at position #1 with {\arrow{>}}},postaction={decorate}}}
\tikzset{-<-/.style={decoration={
  markings,
  mark=at position #1 with {\arrow{<}}},postaction={decorate}}}

\crefname{defin}{Definition}{Definitions}
\crefname{eg}{Example}{Examples}
\crefname{lem}{Lemma}{Lemmas}
\crefname{theo}{Theorem}{Theorems}
\crefname{equation}{}{}
\crefname{enumi}{}{}

\newcommand\C{\mathbb{C}}
\newcommand\Z{\mathbb{Z}}

\newcommand\N{\mathbb{N}}
\newcommand\kk{\Bbbk}

\newcommand\AH{\mathcal{AH}}
\newcommand\AW{\mathcal{AW}}

\newcommand\cC{\mathcal{C}}
\newcommand\cB{\mathcal{B}}
\newcommand\cH{\mathcal{H}}
\newcommand\Heis{\mathcal{H}\mathit{eis}}

\newcommand\cS{\mathcal{S}}
\newcommand\cW{\mathcal{W}}

\newcommand\one{\mathbb{1}}

\newcommand\dg{\mathrm{deg}}

\newcommand\Id{\mathrm{Id}}

\newcommand\tr{\mathrm{tr}}
\newcommand\Vect{\mathrm{Vect}}

\newcommand\fgl{\mathfrak{gl}}
\newcommand\fh{\mathfrak{h}}

\newcommand\md{\textup{-mod}}

\def\chk#1{#1^{\smash{\scalebox{.8}[1.4]{\rotatebox{90}{\textnormal{\guilsinglleft}}}}}} 

\DeclareMathOperator{\Add}{Add}

\DeclareMathOperator{\End}{End}

\DeclareMathOperator{\Hom}{Hom}

\DeclareMathOperator{\Iso}{Iso}
\DeclareMathOperator{\Kar}{Kar}     

\DeclareMathOperator{\Span}{Span}
\DeclareMathOperator{\Sym}{Sym}
\DeclareMathOperator{\Tr}{Tr}

\newtheorem{theo}{Theorem}[section]
\newtheorem{prop}[theo]{Proposition}
\newtheorem{lem}[theo]{Lemma}
\newtheorem*{lem*}{Lemma}

\theoremstyle{definition}
\newtheorem{defin}[theo]{Definition}
\newtheorem{rem}[theo]{Remark}
\newtheorem{eg}[theo]{Example}

\numberwithin{equation}{section}
\allowdisplaybreaks

\setenumerate[1]{label=(\alph*)}  
\setenumerate[2]{label=(\roman*)}

\newtoggle{comments}
\newtoggle{details}
\newtoggle{detailsnote}

\iftoggle{comments}{%
  \newcommand{\comments}[1]{
    \ \\
    {\color{red}
      \textbf{AS:} #1
    }
    \\
    }
}{%
  \newcommand{\comments}[1]{}
}

\iftoggle{details}{%
  \newcommand{\details}[1]{
      \ \\
      {\color{OliveGreen}
        \textbf{Details:} #1
      }
      \\
  }
}{%
  \newcommand{\details}[1]{}
}

\begin{document}
%

\title{String diagrams and categorification}

\author{Alistair Savage}
\address{
  Department of Mathematics and Statistics \\
  University of Ottawa
}
\urladdr{\href{https://alistairsavage.ca}{alistairsavage.ca}, \textrm{\textit{ORCiD}:} \href{https://orcid.org/0000-0002-2859-0239}{orcid.org/0000-0002-2859-0239}}
\email{alistair.savage@uottawa.ca}
\thanks{This work was supported by Discovery Grant RGPIN-2017-03854 from the Natural Sciences and Engineering Research Council of Canada.}

\begin{abstract}
  These are lectures notes for a mini-course given at the conference \emph{Interactions of Quantum Affine Algebras with Cluster Algebras, Current Algebras, and Categorification} in June 2018.  The goal is to introduce the reader to string diagram techniques for monoidal categories, with an emphasis on their role in categorification.
\end{abstract}

\dedicatory{Dedicated to Vyjayanthi Chari on the occasion of her 60th birthday}

\maketitle
\thispagestyle{empty}

\tableofcontents

\section{Introduction}

Categorification is rapidly becoming a fundamental concept in many areas of mathematics, including representation theory, topology, algebraic combinatorics, and mathematical physics.  One of the principal ingredients in categorification is the notion of a monoidal category.  The goal of these notes is to introduce the reader to these categories as they often appear in categorification.  Our intention is to motivate the definitions as much as possible, to help the reader build an intuitive understanding of the underlying concepts.

We begin in \cref{sec:defs} with the definition of a strict $\kk$-linear monoidal category.  Our treatment will center around the string diagram calculus for such categories.  The importance of this formalism comes from both the geometric intuition it provides and the fact that string diagrams are the framework upon which the applications of categorification to other areas such as knot theory and topology are built.

In \cref{sec:monalg} give a number of examples of strict $\kk$-linear monoidal categories.  Most mathematicians encounter monoidal categories as additional structure on some concept they already study: sets, vector spaces, group representations, etc.\ all naturally form monoidal categories.  However, we will define abstract monoidal categories via generators and relations.  Even though this idea has been around for some time, it is still somewhat foreign to many mathematicians working outside of category theory.  We will see how, using this approach, one can obtain extremely efficient descriptions of familiar objects such as symmetric groups, degenerate affine Hecke algebras, braid groups, and wreath product algebras.

The formalism of string diagrams is at its best when one has a \emph{pivotal category}, and we turn to this concept in \cref{sec:piv}.  We start by discussing dual objects in monoidal categories.  Pivotal categories are categories in which all objects have duals, the duality data is compatible with the tensor product, and the right and left mates of morphisms are equal.  In pivotal categories, isotopic string diagrams correspond to the same morphism, allowing for intuitive topological arguments, as well as deep connections to topology and knot theory.

In \cref{sec:cat}, we discuss the idea of categorification, beginning with what is perhaps the standard approach, involving the Grothendieck group/ring of an additive category.  We then discuss the trace of a category, and how this gives rise to a second type of categorification.  We also define the Chern character map, which relates the Grothendieck group to the trace, and the notion of idempotent completion, motivated by the concept of a projective module.

We conclude, in \cref{sec:Heiscat}, with the example of Heisenberg categories.  We define these categories using the ideas we have developed, and explain their relationship to the Heisenberg algebra.  Our discussion here is necessarily brief, aiming only to give the reader a taste of a current area of research.  We point the interested reader to references for further reading.

\section{Strict monoidal categories and string diagrams\label{sec:defs}}

\subsection{Definitions}

Throughout this document, all categories are assumed to be locally small.  In other words, we have a \emph{set} of morphisms between any two objects.

A \emph{strict monoidal category} is a category $\cC$ equipped with
\begin{itemize}
  \item a bifunctor (the \emph{tensor product}) $\otimes \colon \cC \times \cC \to \cC$ and
  \item a \emph{unit object} $\one$
\end{itemize}
such that, for all objects $X$, $Y$, and $Z$ of $\cC$, we have
\begin{itemize}
  \item $(X \otimes Y) \otimes Z = X \otimes (Y \otimes Z)$ and
  \item $\one \otimes X = X = X \otimes \one$,
\end{itemize}
and, for all morphisms $f$, $g$, and $h$ of $\cC$, we have
\begin{itemize}
  \item $(f \otimes g) \otimes h = f \otimes (g \otimes h)$ and
  \item $1_\one \otimes f = f = f \otimes 1_\one$.
\end{itemize}
Here, and throughout the document, $1_X$ denotes the identity endomorphism of an object $X$.

\begin{rem} \label{strictify}
  Note that, in a (not necessarily strict) \emph{monoidal category}, the equalities above are replaced by isomorphism, and one imposes certain coherence conditions.  For example, suppose $\kk$ is a field, and let $\Vect_\kk$ be the category of finite-dimensional $\kk$-vector spaces.  In this category one has isomorphisms $(U \otimes V) \otimes W \cong U \otimes (V \otimes W)$, but these isomorphisms are not equalities in general.  Similarly, the unit object in this category is the one-dimensional vector space $\kk$, and we have $\kk \otimes V \cong V \cong V \otimes \kk$ for any vector space $V$.

  We will be building monoidal categories ``from scratch'' via generators and relations.  Thus, we are free to require them to be strict.  In general, Mac Lane's coherence theorem for monoidal categories asserts that every monoidal category is monoidally equivalent to a strict one.  (For a proof of this fact, see \cite[\S VII.2]{Mac98} or \cite[\S XI.5]{Kas95}.)  So, in practice, we don't lose much by assuming monoidal categories are strict.  (See also \cite{Sch01}.)
\end{rem}

Fix a commutative ground ring $\kk$.  A \emph{$\kk$-linear category} is a category $\cC$ such that
\begin{itemize}
  \item for any two objects $X$ and $Y$ of $\cC$, the hom-set $\Hom_\cC(X,Y)$ is a $\kk$-module,
  \item composition of morphisms is bilinear:
    \begin{gather*}
      f \circ (\alpha g + \beta h) = \alpha (f \circ g) + \beta (f \circ h), \\
      (\alpha f + \beta g) \circ h = \alpha (f \circ h) + \beta (g \circ h),
    \end{gather*}
    for all $\alpha,\beta \in \kk$ and morphisms $f$, $g$, and $h$ such that the above operations are defined.
\end{itemize}
The category of $\kk$-modules is an example of a $\kk$-linear category.  For any two $\kk$-modules $M$ and $N$, the space $\Hom_\kk(M,N)$ is again a $\kk$-module under the usual pointwise operations.  Composition is bilinear with respect to this $\kk$-module structure.

A \emph{strict $\kk$-linear monoidal category} is a category that is both strict monoidal and $\kk$-linear, and such that the tensor product of morphisms is $\kk$-bilinear.  Before discussing some examples, we mention the important \emph{interchange law}.  Suppose
\[
  X_1 \xrightarrow{f} X_2
  \qquad \text{and} \qquad
  Y_1 \xrightarrow{g} Y_2
\]
are morphisms in a strict $\kk$-linear monoidal category $\cC$.  Then
\[
  (1_{X_2} \otimes g) \circ (f \otimes 1_{Y_1})
  = \otimes ((1_{X_2},g)) \circ \otimes ((f,1_{Y_1}))
  = \otimes ((1_{X_2},g) \circ (f,1_{Y_1}))
  = \otimes ((f,g))
  = f \otimes g,
\]
where the second equality uses that the tensor product is a bifunctor.  Similarly,
\[
  (f \otimes 1_{Y_2}) \circ (1_{X_1} \circ g) = f \otimes g.
\]
Thus, the following diagram commutes:
\[
  \xymatrix{
    X_1 \otimes Y_1 \ar[r]^{1 \otimes g} \ar[d]_{f \otimes 1} \ar[dr]^{f \otimes g} & X_1 \otimes Y_2 \ar[d]^{f \otimes 1} \\
    X_2 \otimes Y_1 \ar[r]_{1 \otimes g} & X_2 \otimes Y_2
  }
\]

\subsection{Examples\label{centereg}}

Let's consider a very simple strict monoidal category.   Every monoidal category must have a unit object $\one$ by definition.  But it is possible that this is the \emph{only} object.  The identity axiom for a strict monoidal category forces $\one \otimes \one = \one$.  There is only one hom-set in this category, namely
\[
  \End(\one) := \Hom(\one,\one).
\]
The associativity axiom for morphisms in a category implies that $\End(\one)$ is a monoid under composition, with identity $1_\one$, the identity endomorphism of $\one$.  The axioms of a strict monoidal category imply that $\End(\one)$ is also a monoid under the tensor product.  However, the interchange law forces these monoids to coincide, and to be commutative!  Indeed, for all $f,g \in \End(\one)$, we have
\begin{equation} \label{swirl}
  f \circ g
  = (f \otimes 1_\one) \circ(1_\one \otimes g)
  = f \otimes g
  = (1_\one \otimes g) \circ (f \otimes 1_\one)
  = g \circ f.
\end{equation}
Conversely, given any commutative monoid $A$, we have a strict monoidal category with one object $\one$, and $\End(\one) = A$.  The composition and tensor product are both given by the multiplication in $A$.

Now consider a strict \emph{$\kk$-linear} monoidal category with one object $\one$.  Then $\End(\one)$ is an associative $\kk$-algebra, and an argument exactly analogous to the one above shows that it is, in fact, commutative.  Conversely, every commutative associative $\kk$-algebra gives rise to a one-object strict $\kk$-linear monoidal category.

Note that the above discussion actually shows that $\End(\one)$ is a commutative monoid in \emph{any} strict monoidal category and is a commutative $\kk$-algebra in \emph{any} strict $\kk$-linear monoidal category.  The monoid/algebra $\End(\one)$ is called the \emph{center} of the category.

\begin{eg}[Center of $\Vect_\kk$] \label{centerVect}
  Suppose $\kk$ is a field and consider the category $\Vect_\kk$ of finite-dimensional $\kk$-vector spaces.  This is not a \emph{strict} monoidal category, but, as noted in \cref{strictify} (see, in particular, \cite[Th.~4.3]{Sch01}), we can safely avoid this technicality.  The unit object of $\Vect_\kk$ is the one-dimensional vector space $\kk$, and so the center of this category is $\End_\kk (\kk)$, which is canonically isomorphic, as a ring, to $\kk$ via the isomorphism
  \begin{equation} \label{Endk}
    \End_\kk(\kk) \xrightarrow{\cong} \kk,\quad
    f \mapsto f(1).
  \end{equation}
\end{eg}

\subsection{String diagrams}

Strict monoidal categories are especially well suited to being depicted using the language of \emph{string diagrams}.  These diagrams, which are sometimes also called \emph{Penrose diagrams}, have their origins in work of Roger Penrose in physics \cite{Pen71}.  Working with strings diagrams helps build intuition.  It also often makes certain arguments obvious, whereas the corresponding algebraic proof can be a bit opaque.  We give here a brief overview of string diagrams, referring the reader to \cite[Ch.~2]{TV17} for a detailed treatment.  Throughout this section, $\cC$ will denote a strict $\kk$-linear monoidal category.

We will denote a morphism $f \colon X \to Y$ by a strand with a coupon labeled $f$:
\[
  \begin{tikzpicture}[anchorbase]
    \draw (0,0) node[anchor=north] {\regionlabel{X}} to (0,1) node[anchor=south] {\regionlabel{Y}};
    \filldraw[black,fill=white] (0,0.7) arc(90:450:0.2);
    \node at (0,0.5) {\tokenlabel{f}};
  \end{tikzpicture}
\]
Note that we are adopting the convention that diagrams should be read from bottom to top.  The \emph{identity map} $1_X \colon X \to X$ is a string with no coupon:
\[
  \begin{tikzpicture}[anchorbase]
    \draw (0,0) node[anchor=north] {\regionlabel{X}} to (0,1) node[anchor=south] {\regionlabel{X}};
  \end{tikzpicture}
\]
We sometimes omit the object labels (e.g.\ $X$ and $Y$ above) when they are clear or unimportant.  We will also sometimes distinguish identity maps of different objects by some sort of decoration of the string (orientation, dashed versus solid, etc.), rather than by adding object labels.

Composition is denoted by \emph{vertical stacking} (recall that we read pictures bottom to top) and tensor product is \emph{horizontal juxtaposition}:
\[
  \begin{tikzpicture}[anchorbase]
    \draw (0,0) to (0,1.4);
    \filldraw[black,fill=white] (0,1.2) arc(90:450:0.2);
    \node at (0,1) {\tokenlabel{f}};
    \filldraw[black,fill=white] (0,0.6) arc(90:450:0.2);
    \node at (0,0.4) {\tokenlabel{g}};
  \end{tikzpicture}
  \ =\
  \begin{tikzpicture}[anchorbase]
    \draw (0,0) to (0,1.4);
    \filldraw[black,fill=white] (0,1) arc(90:450:0.3);
    \node at (0,0.7) {\tokenlabel{f \circ g}};
  \end{tikzpicture}
  \qquad \text{and} \qquad
  \begin{tikzpicture}[anchorbase]
    \draw (0,0) to (0,1);
    \filldraw[black,fill=white] (0,0.7) arc(90:450:0.2);
    \node at (0,0.5) {\tokenlabel{f}};
  \end{tikzpicture}
  \ \otimes \
  \begin{tikzpicture}[anchorbase]
    \draw (0,0) to (0,1);
    \filldraw[black,fill=white] (0,0.7) arc(90:450:0.2);
    \node at (0,0.5) {\tokenlabel{g}};
  \end{tikzpicture}
  \ =\
  \begin{tikzpicture}[anchorbase]
    \draw (0,0) to (0,1);
    \filldraw[black,fill=white] (0,0.7) arc(90:450:0.2);
    \node at (0,0.5) {\tokenlabel{f}};
    \draw (0.5,0) to (0.5,1);
    \filldraw[black,fill=white] (0.5,0.7) arc(90:450:0.2);
    \node at (0.5,0.5) {\tokenlabel{g}};
  \end{tikzpicture}
  \ .
\]
The \emph{interchange law} then becomes:
\[
  \begin{tikzpicture}[anchorbase]
    \draw (0,0) to (0,1.6);
    \filldraw[black,fill=white] (0,1.3) arc(90:450:0.2);
    \node at (0,1.1) {\tokenlabel{f}};
    \draw (0.5,0) to (0.5,1.6);
    \filldraw[black,fill=white] (0.5,0.7) arc(90:450:0.2);
    \node at (0.5,0.5) {\tokenlabel{g}};
  \end{tikzpicture}
  \ =\
  \begin{tikzpicture}[anchorbase]
    \draw (0,0) to (0,1.6);
    \filldraw[black,fill=white] (0,1) arc(90:450:0.2);
    \node at (0,0.8) {\tokenlabel{f}};
    \draw (0.5,0) to (0.5,1.6);
    \filldraw[black,fill=white] (0.5,1) arc(90:450:0.2);
    \node at (0.5,0.8) {\tokenlabel{g}};
  \end{tikzpicture}
  \ =\
  \begin{tikzpicture}[anchorbase]
    \draw (0,0) to (0,1.6);
    \filldraw[black,fill=white] (0,0.7) arc(90:450:0.2);
    \node at (0,0.5) {\tokenlabel{f}};
    \draw (0.5,0) to (0.5,1.6);
    \filldraw[black,fill=white] (0.5,1.3) arc(90:450:0.2);
    \node at (0.5,1.1) {\tokenlabel{g}};
  \end{tikzpicture}
\]
A general morphism $f \colon X_1 \otimes \dotsb \otimes X_n \to Y_1 \otimes \dotsb \otimes Y_m$ can be depicted as a coupon with $n$ strands emanating from the bottom and $m$ strands emanating from the top:
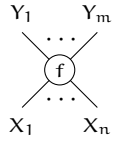
\[
  \begin{tikzpicture}
    \draw (-0.5,-0.5) node[anchor=north] {\regionlabel{X_1}} to (0.5,0.5) node[anchor=south] {\regionlabel{Y_m}};
    \draw (0.5,-0.5) node[anchor=north] {\regionlabel{X_n}} to (-0.5,0.5) node[anchor=south] {\regionlabel{Y_1}};
    \filldraw[black,fill=white] (0,0.2) arc(90:450:0.2);
    \node at (0,0) {\tokenlabel{f}};
    \node at (0.05,-0.4) {$\cdots$};
    \node at (0.05,0.4) {$\cdots$};
  \end{tikzpicture}
\]

For the moment, let us denote the identity morphism $1_\one$ of the identity object $\one$ by a dashed line:
\[
  \begin{tikzpicture}[anchorbase]
    \draw[dashed] (0,0) to (0,1);
  \end{tikzpicture}
\]
Then the translation into diagrams of our argument from \cref{centereg} that the center $\End(\one)$ of the category is a commutative algebra becomes that, for all $f,g \in \End(\one)$,
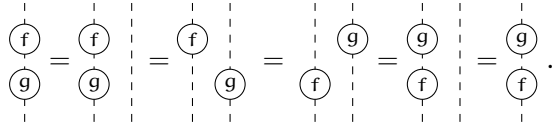
\[
  \begin{tikzpicture}[anchorbase]
    \draw[dashed] (0,0) to (0,1.6);
    \filldraw[black,fill=white] (0,1.3) arc(90:450:0.2);
    \node at (0,1.1) {\tokenlabel{f}};
    \filldraw[black,fill=white] (0,0.7) arc(90:450:0.2);
    \node at (0,0.5) {\tokenlabel{g}};
  \end{tikzpicture}
  =
  \begin{tikzpicture}[anchorbase]
    \draw[dashed] (0,0) to (0,1.6);
    \filldraw[black,fill=white] (0,1.3) arc(90:450:0.2);
    \node at (0,1.1) {\tokenlabel{f}};
    \filldraw[black,fill=white] (0,0.7) arc(90:450:0.2);
    \node at (0,0.5) {\tokenlabel{g}};
    \draw[dashed] (0.5,0) to (0.5,1.6);
  \end{tikzpicture}
  \ =
  \begin{tikzpicture}[anchorbase]
    \draw[dashed] (0,0) to (0,1.6);
    \filldraw[black,fill=white] (0,1.3) arc(90:450:0.2);
    \node at (0,1.1) {\tokenlabel{f}};
    \draw[dashed] (0.5,0) to (0.5,1.6);
    \filldraw[black,fill=white] (0.5,0.7) arc(90:450:0.2);
    \node at (0.5,0.5) {\tokenlabel{g}};
  \end{tikzpicture}
  \ =\
  \begin{tikzpicture}[anchorbase]
    \draw[dashed] (0,0) to (0,1.6);
    \filldraw[black,fill=white] (0,0.7) arc(90:450:0.2);
    \node at (0,0.5) {\tokenlabel{f}};
    \draw[dashed] (0.5,0) to (0.5,1.6);
    \filldraw[black,fill=white] (0.5,1.3) arc(90:450:0.2);
    \node at (0.5,1.1) {\tokenlabel{g}};
  \end{tikzpicture}
  =
  \begin{tikzpicture}[anchorbase]
    \draw[dashed] (0,0) to (0,1.6);
    \filldraw[black,fill=white] (0,1.3) arc(90:450:0.2);
    \node at (0,1.1) {\tokenlabel{g}};
    \filldraw[black,fill=white] (0,0.7) arc(90:450:0.2);
    \node at (0,0.5) {\tokenlabel{f}};
    \draw[dashed] (0.5,0) to (0.5,1.6);
  \end{tikzpicture}
  \ =
  \begin{tikzpicture}[anchorbase]
    \draw[dashed] (0,0) to (0,1.6);
    \filldraw[black,fill=white] (0,1.3) arc(90:450:0.2);
    \node at (0,1.1) {\tokenlabel{g}};
    \filldraw[black,fill=white] (0,0.7) arc(90:450:0.2);
    \node at (0,0.5) {\tokenlabel{f}};
  \end{tikzpicture}
  \ .
\]

In fact, as we see above, the axioms of a strict ($\kk$-linear) monoidal category make it natural to omit the identity morphism $1_\one$ of the identity object.  So we draw endomorphisms $f \in \End(\one)$ of the identity as free-floating coupons:
\[
  \begin{tikzpicture}[anchorbase]
    \draw (0,0.2) arc(90:450:0.2);
    \node at (0,0) {\tokenlabel{f}};
  \end{tikzpicture}
\]
By \cref{swirl}, the horizontal and vertical juxtaposition of such free-floating coupons coincide.  So we may slide these coupons around at will.

\section{Monoidally generated algebras\label{sec:monalg}}

\subsection{Presentations}

One should think of working with strict $\kk$-linear monoidal categories as doing ``two-dimensional'' algebra with the morphisms.  Besides the addition (which corresponds to formal addition of string diagrams), we have two flavors of ``multiplication'': horizontal (the tensor product) and vertical (composition in the category).

Just as one can define associative algebras via generators and relations, one can also define strict $\kk$-linear monoidal categories in this way.  Recall that the free associative $\kk$-algebra $A$ on some set $\{a_i : i \in I\}$ of generators consists of formal finite $\kk$-linear combinations of words in the generators.  These words are of the form
\[
  a_{i_1} a_{i_2} \dotsm a_{i_n},\quad i_1,i_2,\dotsc,i_n \in I.
\]
Multiplication is given by concatenation of words, extended by linearity.  The empty word corresponds to the multiplicative identity.  If we wish to impose some set $R \subseteq A$ of relations, we then consider the algebra $A/\langle R \rangle$, where $\langle R \rangle$ is the two-sided ideal of $A$ generated by the set $R$.  What this means in practice is that we can make ``local substitutions'' in words using the relations.  For example, if $A$ is the algebra with generators $a$, $b$, $c$, and $d$, and relations $ab=c$, $d^2=ba$, then $R = \{ab-c,d^2-ba\}$, and we have
\[
  acabbcdda = ac(ab)bc(d^2)a = accbcbaa.
\]

In a similar way, we can give presentations of strict $\kk$-linear monoidal categories.  (See \cite[\S I.4.2]{Tur16} for more details.)  Now we should specify a set of generating objects, a set of generating morphisms, and some relations \emph{on morphisms} (not on objects!).  If $\{X_i : i \in I\}$ is our set of generating objects, then an arbitrary object in our category is a finite tensor product of these generating objects:
\[
  X_{i_1} \otimes X_{i_2} \otimes \dotsb \otimes X_{i_n},\quad i_1,i_2,\dotsc,i_n \in I,\ n \in \N.
\]
We think of $\one$ as being the ``empty tensor product''.  If $\{f_j : j \in J\}$ is our set of generating morphisms, then we can take arbitrary tensor products and compositions (when domains and codomains match) of these generators, e.g.
\[
  (f_{j_1} \otimes f_{j_2}) \circ \big( (f_{j_3} \circ f_{j_4}) \otimes f_{j_5} \big),\quad j_1,j_2,j_3,j_4,j_5 \in J.
\]

Working with string diagrams, our generating morphisms are diagrams, and we can vertically and horizontally compose them in any way that makes sense (i.e.\ making sure that domains and codomains match in vertical composition).  Relations allow us to make ``local changes'' in our diagrams.

In the examples to be considered below, we will often have generating objects that we will denote by $\uparrow$ and $\downarrow$.  We will always draw their identity morphisms as
\[
  \begin{tikzpicture}[anchorbase]
    \draw[->] (0,0) to (0,0.5);
  \end{tikzpicture}
  \qquad \text{and} \qquad
  \begin{tikzpicture}[anchorbase]
    \draw[<-] (0,0) to (0,0.5);
  \end{tikzpicture}
  \ ,
\]
respectively.

\subsection{The symmetric group\label{Scat}}

As a concrete example, define $\cS$ to be the strict $\kk$-linear monoidal category with:
\begin{itemize}
  \item one generating object $\uparrow$,

  \item one generating morphism
    \[
      \begin{tikzpicture}[anchorbase]
        \draw[->] (-0.25,-0.25) to (0.25,0.25);
        \draw[->] (0.25,-0.25) to (-0.25,0.25);
      \end{tikzpicture}
      \ \colon \uparrow \otimes \uparrow\ \to\ \uparrow \otimes \uparrow,
    \]

  \item two relations
    \begin{equation} \label{Sn-strings}
      \begin{tikzpicture}[anchorbase]
        \draw[->] (0.3,0) to[out=up,in=down] (-0.3,0.6) to[out=up,in=down] (0.3,1.2);
        \draw[->] (-0.3,0) to[out=up,in=down] (0.3,0.6) to[out=up,in=down] (-0.3,1.2);
      \end{tikzpicture}
      \ =\
      \begin{tikzpicture}[anchorbase]
        \draw[->] (-0.2,0) -- (-0.2,1.2);
        \draw[->] (0.2,0) -- (0.2,1.2);
      \end{tikzpicture}
      \qquad \text{and} \qquad
      \begin{tikzpicture}[anchorbase]
        \draw[->] (0.4,0) -- (-0.4,1.2);
        \draw[->] (0,0) to[out=up, in=down] (-0.4,0.6) to[out=up,in=down] (0,1.2);
        \draw[->] (-0.4,0) -- (0.4,1.2);
      \end{tikzpicture}
      \ =\
      \begin{tikzpicture}[anchorbase]
        \draw[->] (0.4,0) -- (-0.4,1.2);
        \draw[->] (0,0) to[out=up, in=down] (0.4,0.6) to[out=up,in=down] (0,1.2);
        \draw[->] (-0.4,0) -- (0.4,1.2);
      \end{tikzpicture}
      \ .
    \end{equation}
\end{itemize}
One could write these relations in a more traditional algebraic manner, if so desired.  For example, if we let
\[
  s =
  \begin{tikzpicture}[anchorbase]
    \draw[->] (-0.25,-0.25) to (0.25,0.25);
    \draw[->] (0.25,-0.25) to (-0.25,0.25);
  \end{tikzpicture}
  \ \colon \uparrow \otimes \uparrow\ \to\ \uparrow \otimes \uparrow,
\]
then the two relations \cref{Sn-strings} become
\[
  s^2 = 1_{\uparrow \otimes \uparrow}
  \qquad \text{and} \qquad
  (s \otimes 1_\uparrow) \circ (1_\uparrow \otimes s) \circ (s \otimes 1_\uparrow)
  = (1_\uparrow \otimes s) \circ (s \otimes 1_\uparrow) \circ (1_\uparrow \otimes s).
\]

Now, in any $\kk$-linear category (monoidal or not), we have an endomorphism algebra $\End(X)$ of any object $X$.  The multiplication in this algebra is given by vertical composition.  In $\cS$, every object is of the form $\uparrow^{\otimes n}$ for some $n=0,1,2,\dotsc$.  An example of an endomorphism of $\uparrow^{\otimes 4}$ is
\[
  \begin{tikzpicture}[anchorbase]
    \draw[->] (-0.9,0) to[out=up,in=down] (-0.3,0.6) to[out=up,in=down] (-0.9,1.2) to[out=up,in=down] (-0.9,3) to[out=up,in=down] (-0.3,3.6);
    \draw[->] (-0.3,0) to[out=up,in=down] (-0.9,0.6) to[out=up,in=down] (-0.3,1.2) to[out=up,in=down] (0.3,1.8) to[out=up,in=down] (0.9,2.4) to[out=up,in=down] (0.9,3) to[out=up,in=down] (0.3,3.6);
    \draw[->] (0.3,0) to[out=up,in=down] (0.3,1.2) to[out=up,in=down] (-0.3,1.8) to[out=up,in=down] (-0.3,2.4) to[out=up,in=down] (0.3,3) to[out=up,in=down] (0.9,3.6);
    \draw[->] (0.9,0) to[out=up,in=down] (0.9,1.8) to[out=up,in=down] (0.3,2.4) to[out=up,in=down] (-0.3,3) to[out=up,in=down] (-0.9,3.6);
  \end{tikzpicture}
  \ +2\
  \begin{tikzpicture}[anchorbase]
    \draw[->] (-0.9,0) to (0.9,2.4);
    \draw[->] (-0.3,0) to (-0.9,2.4);
    \draw[->] (0.3,0) to (-0.3,2.4);
    \draw[->] (0.9,0) to (0.3,2.4);
  \end{tikzpicture}
  \ .
\]
Using the relations, we see that this morphism is equal to
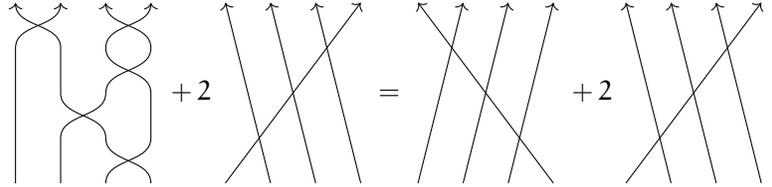
\[
  \begin{tikzpicture}[anchorbase]
    \draw[->] (-0.9,1.2) to[out=up,in=down] (-0.9,3) to[out=up,in=down] (-0.3,3.6);
    \draw[->] (-0.3,1.2) to[out=up,in=down] (-0.3,1.8) to[out=up,in=down] (0.3,2.4) to[out=up,in=down] (0.9,3) to[out=up,in=down] (0.3,3.6);
    \draw[->] (0.3,1.2) to[out=up,in=down] (0.9,1.8) to[out=up,in=down] (0.9,2.4) to[out=up,in=down] (0.3,3) to[out=up,in=down] (0.9,3.6);
    \draw[->] (0.9,1.2) to[out=up,in=down] (0.3,1.8) to[out=up,in=down] (-0.3,2.4) to[out=up,in=down] (-0.3,3) to[out=up,in=down] (-0.9,3.6);
  \end{tikzpicture}
  \ +2\
  \begin{tikzpicture}[anchorbase]
    \draw[->] (-0.9,0) to (0.9,2.4);
    \draw[->] (-0.3,0) to (-0.9,2.4);
    \draw[->] (0.3,0) to (-0.3,2.4);
    \draw[->] (0.9,0) to (0.3,2.4);
  \end{tikzpicture}
  \ =\
  \begin{tikzpicture}[anchorbase]
    \draw[->] (-0.9,0) to (-0.3,2.4);
    \draw[->] (-0.3,0) to (0.3,2.4);
    \draw[->] (0.3,0) to (0.9,2.4);
    \draw[->] (0.9,0) to (-0.9,2.4);
  \end{tikzpicture}
  \ +2\
  \begin{tikzpicture}[anchorbase]
    \draw[->] (-0.9,0) to (0.9,2.4);
    \draw[->] (-0.3,0) to (-0.9,2.4);
    \draw[->] (0.3,0) to (-0.3,2.4);
    \draw[->] (0.9,0) to (0.3,2.4);
  \end{tikzpicture}
  \ .
\]

Fix a positive integer $n$ and recall that the group algebra $\kk S_n$ of the symmetric group on $n$ letters has a presentation with generators $s_1,s_2,\dotsc,s_{n-1}$ (the simple transpositions) and relations
\begin{align}
  s_i^2 &= 1, & 1 \le i \le n-1,  \label{Srel1} \\
  s_i s_{i+1} s_i &= s_{i+1} s_i s_{i+1}, & 1 \le i \le n-2,  \label{Srel2} \\
  s_i s_j &= s_j s_i, & 1 \le i,j \le n-1,\ |i-j| > 1. \label{Srel3}
\end{align}
Consider the map
\[
  \kk S_n \to \End_\cS (\uparrow^{\otimes n})
\]
where $s_i$ is sent to the crossing of the $i$-th and $(i+1)$-st strands, labeled from right to left.  In fact, this map is an isomorphism of algebras.  So the category $\cS$ contains the group algebras of all of the symmetric groups!

Note that the presentation of $\cS$ is much more efficient than the presentation of the algebras $\kk S_n$.  To define $\cS$ we need only \emph{one} generator and \emph{two} relations, as opposed to the $n-1$ generators and relations \cref{Srel1,Srel2,Srel3} (whose number is of order $n^2$) in the algebraic presentation of $\kk S_n$, for \emph{each} $n$.  This efficiency comes from the fact that we are generating the algebras \emph{monoidally}, where we have both vertical and horizontal ``multiplication''.  In particular, the ``distant braid relation'' \cref{Srel3} follows for free from the interchange law:
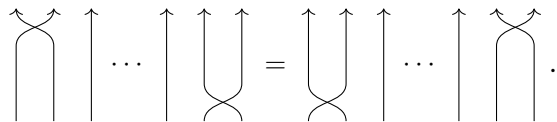
\[
  \begin{tikzpicture}[anchorbase]
    \draw[->] (0,-1) to (0,0) to[out=up,in=down] (0.5,0.5);
    \draw[->] (0.5,-1) to (0.5,0) to[out=up,in=down] (0,0.5);
    \draw[->] (1,-1) to (1,0.5);
    \node at (1.5,-0.25) {$\cdots$};
    \draw[->] (2,-1) to (2,0.5);
    \draw[->] (2.5,-1) to[out=up,in=down] (3,-0.5) to (3,0.5);
    \draw[->] (3,-1) to[out=up,in=down] (2.5,-0.5) to (2.5,0.5);
  \end{tikzpicture}
  \ =\
  \begin{tikzpicture}[anchorbase]
    \draw[->] (0.5,-1) to[out=up,in=down] (0,-0.5) to (0,0.5);
    \draw[->] (0,-1) to[out=up,in=down] (0.5,-0.5) to (0.5,0.5);
    \draw[->] (1,-1) to (1,0.5);
    \node at (1.5,-0.25) {$\cdots$};
    \draw[->] (2,-1) to (2,0.5);
    \draw[->] (2.5,-1) to (2.5,0) to[out=up,in=down] (3,0.5);
    \draw[->] (3,-1) to (3,0) to[out=up,in=down] (2.5,0.5);
  \end{tikzpicture}
  \ .
\]

\subsection{Degenerate affine Hecke algebras\label{sec:dAHA}}

Let $\AH^\dg$ be the strict $\kk$-linear monoidal category $\cS$ defined in \cref{Scat}, but with an additional generating morphism, which we will call a \emph{dot},
\[
  \begin{tikzpicture}[anchorbase]
    \draw[->] (0,0) to (0,0.6);
    \redcircle{(0,0.3)};
  \end{tikzpicture}
  \ \colon \uparrow\ \to\ \uparrow
\]
and one additional relation:
\[
  \begin{tikzpicture}[anchorbase]
    \draw[->] (0,0) -- (0.6,0.6);
    \draw[->] (0.6,0) -- (0,0.6);
    \redcircle{(0.15,.45)};
  \end{tikzpicture}
  \ -\
  \begin{tikzpicture}[anchorbase]
    \draw[->] (0,0) -- (0.6,0.6);
    \draw[->] (0.6,0) -- (0,0.6);
    \redcircle{(.45,.15)};
  \end{tikzpicture}
  \ =\
  \begin{tikzpicture}[anchorbase]
    \draw[->] (0,0) -- (0,0.6);
    \draw[->] (0.3,0) -- (0.3,0.6);
  \end{tikzpicture}\ .
\]
Now
\[
  \End_{\AH^\dg} (\uparrow^{\otimes n})
\]
is isomorphic to the \emph{degenerate affine Hecke algebra} of type $A_{n-1}$.  In the category $\AH^\dg$, the endomorphism algebra of $\uparrow$ is now infinite dimensional, with basis given by
\[
  \begin{tikzpicture}[anchorbase]
    \draw[->] (0,0) to (0,1.1);
    \redcircle{(0,0.2)};
    \redcircle{(0,0.4)};
    \redcircle{(0,0.6)};
    \redcircle{(0,0.8)};
    \draw (0,0.5) node[anchor=west] {$\bigg\} m$ dots};
  \end{tikzpicture}
  ,\ m=0,1,2,\dotsc.
\]

\subsection{The braid group\label{braidcat}}

Consider another strict $\kk$-linear monoidal category $\cB$ with one generating object $\uparrow$ and one generating morphism
\begin{equation} \label{poscross}
  \begin{tikzpicture}[anchorbase]
    \draw[->] (0.25,-0.25) to (-0.25,0.25);
    \draw[wipe] (-0.25,-0.25) to (0.25,0.25);
    \draw[->] (-0.25,-0.25) to (0.25,0.25);
  \end{tikzpicture}
  \ \colon \uparrow \otimes \uparrow\ \to\ \uparrow \otimes \uparrow.
\end{equation}
We want to impose the relation that this morphism is invertible.  Thus we add another generating morphism which is inverse to \cref{poscross}.  Let us denote this inverse by
\begin{equation} \label{negcross}
  \begin{tikzpicture}[anchorbase]
    \draw[->] (-0.25,-0.25) to (0.25,0.25);
    \draw[wipe] (0.25,-0.25) to (-0.25,0.25);
    \draw[->] (0.25,-0.25) to (-0.25,0.25);
  \end{tikzpicture}
  \ \colon \uparrow \otimes \uparrow\ \to\ \uparrow \otimes \uparrow.
\end{equation}
To say that \cref{negcross} and \cref{poscross} are inverse means that we impose the relations
\begin{equation} \label{braid-inv}
  \begin{tikzpicture}[anchorbase]
    \draw[->] (0.3,0) to[out=up,in=down] (-0.3,0.6) to[out=up,in=down] (0.3,1.2);
    \draw[wipe] (-0.3,0) to[out=up,in=down] (0.3,0.6) to[out=up,in=down] (-0.3,1.2);
    \draw[->] (-0.3,0) to[out=up,in=down] (0.3,0.6) to[out=up,in=down] (-0.3,1.2);
  \end{tikzpicture}
  \ =\
  \begin{tikzpicture}[anchorbase]
    \draw[->] (-0.2,0) -- (-0.2,1.2);
    \draw[->] (0.2,0) -- (0.2,1.2);
  \end{tikzpicture}
  \qquad \text{and} \qquad
  \begin{tikzpicture}[anchorbase]
    \draw[->] (-0.3,0) to[out=up,in=down] (0.3,0.6) to[out=up,in=down] (-0.3,1.2);
    \draw[wipe] (0.3,0) to[out=up,in=down] (-0.3,0.6) to[out=up,in=down] (0.3,1.2);
    \draw[->] (0.3,0) to[out=up,in=down] (-0.3,0.6) to[out=up,in=down] (0.3,1.2);
  \end{tikzpicture}
  \ =\
  \begin{tikzpicture}[anchorbase]
    \draw[->] (-0.2,0) -- (-0.2,1.2);
    \draw[->] (0.2,0) -- (0.2,1.2);
  \end{tikzpicture}\ .
\end{equation}
To complete the definition of $\cB$, we impose one more relation:
\begin{equation} \label{braid}
  \begin{tikzpicture}[anchorbase]
    \draw[->] (0.4,0) -- (-0.4,1.2);
    \draw[wipe] (0,0) to[out=up, in=down] (-0.4,0.6) to[out=up,in=down] (0,1.2);
    \draw[->] (0,0) to[out=up, in=down] (-0.4,0.6) to[out=up,in=down] (0,1.2);
    \draw[wipe] (-0.4,0) -- (0.4,1.2);
    \draw[->] (-0.4,0) -- (0.4,1.2);
  \end{tikzpicture}
  \ =\
  \begin{tikzpicture}[anchorbase]
    \draw[->] (0.4,0) -- (-0.4,1.2);
    \draw[wipe] (0,0) to[out=up, in=down] (0.4,0.6) to[out=up,in=down] (0,1.2);
    \draw[->] (0,0) to[out=up, in=down] (0.4,0.6) to[out=up,in=down] (0,1.2);
    \draw[wipe] (-0.4,0) -- (0.4,1.2);
    \draw[->] (-0.4,0) -- (0.4,1.2);
  \end{tikzpicture}
  \ .
\end{equation}
Then $\End_\cB(\uparrow^{\otimes n})$ is isomorphic to the group algebra of the braid group on $n$ strands.  Again, we see that generating these algebras monoidally is extremely efficient.

\subsection{Hecke algebras}

Fix $z \in \kk^\times$.  Let $\cH(z)$ be the strict $\kk$-linear monoidal category $\cB$ defined in \cref{braidcat}, but with one more relation:
\begin{equation} \label{skein}
  \begin{tikzpicture}[anchorbase]
    \draw[->] (0.25,-0.25) to (-0.25,0.25);
    \draw[wipe] (-0.25,-0.25) to (0.25,0.25);
    \draw[->] (-0.25,-0.25) to (0.25,0.25);
  \end{tikzpicture}
  \ -\
  \begin{tikzpicture}[anchorbase]
    \draw[->] (-0.25,-0.25) to (0.25,0.25);
    \draw[wipe] (0.25,-0.25) to (-0.25,0.25);
    \draw[->] (0.25,-0.25) to (-0.25,0.25);
  \end{tikzpicture}
  \ = z\
  \begin{tikzpicture}[anchorbase]
    \draw[->] (0.2,-0.25) to (0.2,0.25);
    \draw[->] (-0.2,-0.25) to (-0.2,0.25);
  \end{tikzpicture}
  \ .
\end{equation}
If $z=0$, the relation \cref{skein} forces the generators \cref{poscross,negcross} to be equal.  The relations \cref{braid-inv,braid} then reduce to \cref{Sn-strings}.  Hence $\cH(0)=\cS$.  On the other hand, if $\kk = \C(q)$ and $z = q - q^{-1}$, then $\End_{\cH(z)} (\uparrow^{\otimes n})$ is isomorphic to the \emph{Iwahori--Hecke algebra} of type $A_{n-1}$.

\subsection{Wreath product algebras}

Let $A$ be an associative $\kk$-algebra.  Let's modify the category $\cS$ from \cref{Scat} by adding an endomorphism of $\uparrow$ for each element of $A$.  More precisely, define the \emph{wreath product category} $\cW(A)$ to be the strict $\kk$-linear monoidal category obtained from $\cS$ by adding morphisms such that we have an algebra homomorphipsm
\[
  A \to \End(\uparrow),
  \qquad
  a \mapsto
  \begin{tikzpicture}[anchorbase]
    \draw[->] (0,0) to (0,0.6);
    \bluedot{(0,0.3)} node[anchor=west,color=black] {\dotlabel{a}};
  \end{tikzpicture}
  \ .
\]
In particular, this means that
\begin{equation} \label{dotlin}
  \begin{tikzpicture}[anchorbase]
    \draw[->] (0,0) to (0,0.6);
    \bluedot{(0,0.3)} node[anchor=west,color=black] {\dotlabel{(\alpha a+ \beta b)}};
  \end{tikzpicture}
  = \alpha\
  \begin{tikzpicture}[anchorbase]
    \draw[->] (0,0) to (0,0.6);
    \bluedot{(0,0.3)} node[anchor=west,color=black] {\dotlabel{a}};
  \end{tikzpicture}
  + \beta\
  \begin{tikzpicture}[anchorbase]
    \draw[->] (0,0) to (0,0.6);
    \bluedot{(0,0.3)} node[anchor=west,color=black] {\dotlabel{b}};
  \end{tikzpicture}
  \qquad \text{and} \qquad
  \begin{tikzpicture}[anchorbase]
    \draw[->] (0,0) to (0,1);
    \bluedot{(0,0.3)} node[anchor=west,color=black] {\dotlabel{b}};
    \bluedot{(0,0.6)} node[anchor=west,color=black] {\dotlabel{a}};
  \end{tikzpicture}
  =\
  \begin{tikzpicture}[anchorbase]
    \draw[->] (0,0) to (0,1);
    \bluedot{(0,0.5)} node[anchor=west,color=black] {\dotlabel{ab}};
  \end{tikzpicture}
  \qquad \text{for all } \alpha,\beta \in \kk,\ a,b \in A.
\end{equation}
We call the closed circles appearing in the above diagrams \emph{tokens}.  We then impose the additional relation
\begin{equation} \label{tokslide}
  \begin{tikzpicture}[anchorbase]
    \draw[->] (0,0) -- (1,1);
    \draw[->] (1,0) -- (0,1);
    \bluedot{(.25,.25)} node [anchor=east, color=black] {\dotlabel{a}};
  \end{tikzpicture}
  \ =\
  \begin{tikzpicture}[anchorbase]
    \draw[->](0,0) -- (1,1);
    \draw[->](1,0) -- (0,1);
    \bluedot{(0.75,.75)} node [anchor=east, color=black] {\dotlabel{a}};
  \end{tikzpicture}
  \ ,\quad a \in A.
\end{equation}
As an example of a diagrammatic proof, note that we can compose \cref{tokslide} on the top and bottom with a crossing to obtain
\[
  \begin{tikzpicture}[anchorbase]
    \draw[->] (0,0) -- (1,1);
    \draw[->] (1,0) -- (0,1);
    \bluedot{(.25,.25)} node [anchor=east, color=black] {\dotlabel{a}};
  \end{tikzpicture}
  \ =\
  \begin{tikzpicture}[anchorbase]
    \draw[->](0,0) -- (1,1);
    \draw[->](1,0) -- (0,1);
    \bluedot{(0.75,.75)} node [anchor=east, color=black] {\dotlabel{a}};
  \end{tikzpicture}
  \ \implies\
  \begin{tikzpicture}[anchorbase]
    \draw[->] (0.25,-0.5) to[out=up,in=down] (-0.25,0) to[out=up,in=down] (0.25,0.5) to[out=up,in=down] (-0.25,1);
    \draw[->] (-0.25,-0.5) to[out=up,in=down] (0.25,0) to[out=up,in=down] (-0.25,0.5) to[out=up,in=down] (0.25,1);
    \bluedot{(-0.25,0)} node[anchor=east,color=black] {\dotlabel{a}};
  \end{tikzpicture}
  \ =\
  \begin{tikzpicture}[anchorbase]
    \draw[->] (0.25,-0.5) to[out=up,in=down] (-0.25,0) to[out=up,in=down] (0.25,0.5) to[out=up,in=down] (-0.25,1);
    \draw[->] (-0.25,-0.5) to[out=up,in=down] (0.25,0) to[out=up,in=down] (-0.25,0.5) to[out=up,in=down] (0.25,1);
    \bluedot{(0.25,0.5)} node[anchor=west,color=black] {\dotlabel{a}};
  \end{tikzpicture}
  \ \stackrel{\cref{Sn-strings}}{\implies} \
  \begin{tikzpicture}[anchorbase]
    \draw[->] (0,0) -- (1,1);
    \draw[->] (1,0) -- (0,1);
    \bluedot{(0.25,.75)} node [anchor=north east, color=black] {\dotlabel{a}};
  \end{tikzpicture}
  \ =\
  \begin{tikzpicture}[anchorbase]
    \draw[->] (0,0) -- (1,1);
    \draw[->] (1,0) -- (0,1);
    \bluedot{(.75,.25)} node [anchor=south west, color=black] {\dotlabel{a}};
  \end{tikzpicture}
  \ .
\]
So tokens also slide up-left through crossings.

One can show that
\[
  \End_{\cW(A)} (\uparrow^{\otimes n}) \cong A^{\otimes n} \rtimes S_n,
\]
the $n$-th \emph{wreath product algebra} associated to $A$.  As a $\kk$-module,
\[
  A^{\otimes n} \rtimes S_n = A^{\otimes n} \otimes_\kk \kk S_n.
\]
Multiplication is determined by
\[
  (a_1 \otimes \pi_1)(a_2 \otimes \pi_2)
  = a_1 (\pi_1 \cdot a_2) \otimes \pi_1 \pi_2,
  \quad a_1, a_2 \in A^{\otimes n},\ \pi_1, \pi_2 \in S_n,
\]
where $\pi_1 \cdot a_2$ denotes the natural action of $\pi_1 \in S_n$ on $a_2 \in A^{\otimes n}$ by permutation of the factors.  Note that $\cW(\kk) = \cS$, the symmetric group category.

\subsection{Affine wreath product algebras}

The wreath product category $\cW(A)$ is a generalization of the symmetric group category $\cS$ that depends on a choice of associative $\kk$-algebra $A$.  We can generalize the degenerate affine Hecke category $\AH^\dg$ in a similar way as long as we have some additional structure on $A$.  In particular, we suppose that we have a $\kk$-linear \emph{trace map}
\[
  \tr \colon A \to \kk
\]
and dual bases $B$ and $\{\chk{b} : b \in B\}$ of $A$ such that
\[
  \tr(\chk{a}b) = \delta_{a,b}
  \quad \text{for all } a,b \in B.
\]
An algebra with such a trace map is called a \emph{Frobenius algebra}.  We will assume for simplicity here that the trace map is symmetric:
\[
  \tr(ab) = \tr(ba) \quad \text{for all } a,b \in A.
\]

It is an easy linear algebra exercise to verify that the element
\[
  \sum_{b \in B} b \otimes \chk{b} \in A \otimes A
\]
is independent of the choice of basis $B$.  Also, for all $x \in A$, we have
\begin{equation} \label{teleport}
  \begin{split}
    \sum_{b \in B} bx \otimes \chk{b}
    = \sum_{a,b \in B} \tr(\chk{a}bx)a \otimes \chk{b}
    = \sum_{a,b \in B} a \otimes \tr(\chk{a}bx) \chk{b}
    \\
    = \sum_{a,b \in B} a \otimes \tr(x\chk{a}b) \chk{b}
    = \sum_{a \in B} a \otimes x\chk{a}.
  \end{split}
\end{equation}

We define the \emph{affine wreath product category} $\AW(A)$ to be the strict $\kk$-linear monoidal category obtained from $\cW(A)$ by adding a generating morphism
\[
  \begin{tikzpicture}[anchorbase]
    \draw[->] (0,0) to (0,0.6);
    \redcircle{(0,0.3)};
  \end{tikzpicture}
  \ \colon \uparrow\ \to\ \uparrow
\]
and the additional relations
\begin{equation} \label{AWPA}
  \begin{tikzpicture}[anchorbase]
    \draw[->] (0,0) -- (0.6,0.6);
    \draw[->] (0.6,0) -- (0,0.6);
    \redcircle{(0.15,.45)};
  \end{tikzpicture}
  \ -\
  \begin{tikzpicture}[anchorbase]
    \draw[->] (0,0) -- (0.6,0.6);
    \draw[->] (0.6,0) -- (0,0.6);
    \redcircle{(.45,.15)};
  \end{tikzpicture}
  \ = \sum_{b \in B}
  \begin{tikzpicture}[anchorbase]
    \draw[->] (0,0) -- (0,0.6);
    \draw[->] (0.3,0) -- (0.3,0.6);
    \bluedot{(0,0.3)} node[anchor=east,color=black] {\dotlabel{b}};
    \bluedot{(0.3,0.3)} node[anchor=west,color=black] {\dotlabel{\chk{b}}};
  \end{tikzpicture}
  \qquad \text{and} \qquad
  \begin{tikzpicture}[anchorbase]
    \draw[->] (0,0) to (0,1);
    \redcircle{(0,0.3)};
    \bluedot{(0,0.6)} node[anchor=east,color=black] {\dotlabel{a}};
  \end{tikzpicture}
  \ =
  \begin{tikzpicture}[anchorbase]
    \draw[->] (0,0) to (0,1);
    \redcircle{(0,0.6)};
    \bluedot{(0,0.3)} node[anchor=west,color=black] {\dotlabel{a}};
  \end{tikzpicture}
  \ ,\quad \text{for all } a \in A.
\end{equation}
To motivate this definition, examine what happens if we add a token labeled $a \in A$ to the bottom of the left strand of the diagrams involved in the first relation in \cref{AWPA}.  For the diagrams on the left side, the token slides up to the top of the right strand:
\[
  \begin{tikzpicture}[anchorbase]
    \draw[->] (-0.05,-0.05) -- (0.65,0.65);
    \draw[->] (0.65,-0.05) -- (-0.05,0.65);
    \redcircle{(0.15,.45)};
    \bluedot{(0.15,0.15)} node[anchor=east,color=black] {\dotlabel{a}};
  \end{tikzpicture}
  \ =\
  \begin{tikzpicture}[anchorbase]
    \draw[->] (-0.05,-0.05) -- (0.65,0.65);
    \draw[->] (0.65,-0.05) -- (-0.05,0.65);
    \redcircle{(0.15,.45)};
    \bluedot{(0.45,0.45)} node[anchor=west,color=black] {\dotlabel{a}};
  \end{tikzpicture}
  \qquad \text{and} \qquad
  \begin{tikzpicture}[anchorbase]
    \draw[->] (-0.05,-0.05) -- (0.65,0.65);
    \draw[->] (0.65,-0.05) -- (-0.05,0.65);
    \redcircle{(.45,.15)};
    \bluedot{(0.15,0.15)} node[anchor=east,color=black] {\dotlabel{a}};
  \end{tikzpicture}
  \ =\
  \begin{tikzpicture}[anchorbase]
    \draw[->] (-0.05,-0.05) -- (0.65,0.65);
    \draw[->] (0.65,-0.05) -- (-0.05,0.65);
    \redcircle{(.45,.15)};
    \bluedot{(0.45,0.45)} node[anchor=west,color=black] {\dotlabel{a}};
  \end{tikzpicture}
  \ .
\]
Since diagrams are linear in the token labels (see \cref{dotlin}), \cref{teleport} tells us that the exact same thing happens with the term on the right side of the first relation in \cref{AWPA}:
\[
  \sum_{b \in B}
  \begin{tikzpicture}[anchorbase]
    \draw[->] (0,0) -- (0,1);
    \draw[->] (0.4,0) -- (0.4,1);
    \bluedot{(0,0.5)} node[anchor=east,color=black] {\dotlabel{b}};
    \bluedot{(0.4,0.5)} node[anchor=west,color=black] {\dotlabel{\chk{b}}};
    \bluedot{(0,0.25)} node[anchor=east,color=black] {\dotlabel{a}};
  \end{tikzpicture}
  \ = \sum_{b \in B}
  \begin{tikzpicture}[anchorbase]
    \draw[->] (0,0) -- (0,1);
    \draw[->] (0.4,0) -- (0.4,1);
    \bluedot{(0,0.5)} node[anchor=east,color=black] {\dotlabel{ba}};
    \bluedot{(0.4,0.5)} node[anchor=west,color=black] {\dotlabel{\chk{b}}};
  \end{tikzpicture}
  \ \stackrel{\cref{teleport}}{=} \sum_{b \in B}
  \begin{tikzpicture}[anchorbase]
    \draw[->] (0,0) -- (0,1);
    \draw[->] (0.4,0) -- (0.4,1);
    \bluedot{(0,0.5)} node[anchor=east,color=black] {\dotlabel{b}};
    \bluedot{(0.4,0.5)} node[anchor=west,color=black] {\dotlabel{a\chk{b}}};
  \end{tikzpicture}
  \ = \sum_{b \in B}
  \begin{tikzpicture}[anchorbase]
    \draw[->] (0,0) -- (0,1);
    \draw[->] (0.4,0) -- (0.4,1);
    \bluedot{(0,0.5)} node[anchor=east,color=black] {\dotlabel{b}};
    \bluedot{(0.4,0.5)} node[anchor=west,color=black] {\dotlabel{\chk{b}}};
    \bluedot{(0.4,0.75)} node[anchor=west,color=black] {\dotlabel{a}};
  \end{tikzpicture}
  \ .
\]
Loosely speaking, tokens can ``teleport'' across the sum appearing in the left relation in \cref{AWPA}.

The name \emph{affine wreath product category} comes from the fact that
\[
  \End_{\AW(A)}(\uparrow^{\otimes n})
\]
is isomorphic to an \emph{affine wreath product algebra}.  See \cite{Sav17} for a detailed discussion of these algebras.  Note that $\AW(\kk)$ is the degenerate affine Hecke category $\AH^\dg$.  (Here we take the trace map $\tr \colon \kk \to \kk$ to be the identity.)

\subsection{Quantum affine wreath product algebras}

One can also define affine versions of the Hecke category $\cH(z)$ and generalizations of these categeories depending on a Frobenius algebra.  We refer the reader to \cite{BSW2,BS18b,RS19} for further details.

\section{Pivotal categories\label{sec:piv}}

We give here a brief overview of pivotal categories.  Further details can be found in \cite[\S1.7, \S2.1]{TV17}.

\subsection{Duality\label{duality}}

Suppose a strict monoidal category has two objects $\uparrow$ and $\downarrow$.  Recalling our convention that we do not draw the identity morphism of the unit object $\one$, a morphism $\one \to \downarrow \otimes \uparrow$ would have string diagram
\[
  \begin{tikzpicture}[anchorbase]
    \draw[->] (-0.3,0) to[out=down,in=down,looseness=2] (0.3,0);
  \end{tikzpicture}
  \ \colon \one \to\ \downarrow \otimes \uparrow,
\]
where we may decorate the cup with some symbol if we have more than one such morphism.  The fact that the bottom of the diagram is empty space indicates that the domain of this morphism is the unit object $\one$. Similarly, we can have
\[
  \begin{tikzpicture}[anchorbase]
    \draw[->] (-0.3,0) to[out=up,in=up,looseness=2] (0.3,0);
  \end{tikzpicture}
  \ \colon \uparrow \otimes \downarrow\ \to \one.
\]

We say that $\downarrow$ is \emph{right dual} to $\uparrow$ (and $\uparrow$ is \emph{left dual} to $\downarrow$ ) if we have morphisms
\[
  \begin{tikzpicture}[anchorbase]
    \draw[->] (-0.3,0) to[out=down,in=down,looseness=2] (0.3,0);
  \end{tikzpicture}
  \ \colon \one \to \downarrow \otimes \uparrow.
  \qquad \text{and} \qquad
  \begin{tikzpicture}[anchorbase]
    \draw[->] (-0.3,0) to[out=up,in=up,looseness=2] (0.3,0);
  \end{tikzpicture}
  \ \colon \uparrow \otimes \downarrow \to \one
\]
such that
\begin{equation} \label{rzigzag}
  \begin{tikzpicture}[anchorbase]
    \draw[<-] (0.6,0) to (0.6,0.5) to[out=up,in=up,looseness=2] (0.3,0.5) to[out=down,in=down,looseness=2] (0,0.5) to (0,1);
  \end{tikzpicture}
  \ =\
  \begin{tikzpicture}[anchorbase]
    \draw[<-] (0,0) to (0,1);
  \end{tikzpicture}
  \qquad \text{and} \qquad
  \begin{tikzpicture}[anchorbase]
    \draw[<-] (0.6,1) to (0.6,0.5) to[out=down,in=down,looseness=2] (0.3,0.5) to[out=up,in=up,looseness=2] (0,0.5) to (0,0);
  \end{tikzpicture}
  \ =\
  \begin{tikzpicture}[anchorbase]
    \draw[->] (0,0) to (0,1);
  \end{tikzpicture}
  \ .
\end{equation}
(The above relations are analogous to the unit-counit formulation of adjunction of functors.)  A monoidal category in which every object has both left and right duals is called a \emph{rigid}, or \emph{autonomous}, category.

If $\uparrow$ and $\downarrow$ are both left and right dual to each other, then, in addition to the above, we also have
\[
  \begin{tikzpicture}[anchorbase]
    \draw[<-] (-0.3,0) to[out=down,in=down,looseness=2] (0.3,0);
  \end{tikzpicture}
  \ \colon \one \to \uparrow \otimes \downarrow
  \qquad \text{and} \qquad
  \begin{tikzpicture}[anchorbase]
    \draw[<-] (-0.3,0) to[out=up,in=up,looseness=2] (0.3,0);
  \end{tikzpicture}
  \ \colon \downarrow \otimes \uparrow \to \one
\]
such that
\begin{equation} \label{lzigzag}
  \begin{tikzpicture}[anchorbase]
    \draw[->] (0.6,0) to (0.6,0.5) to[out=up,in=up,looseness=2] (0.3,0.5) to[out=down,in=down,looseness=2] (0,0.5) to (0,1);
  \end{tikzpicture}
  \ =\
  \begin{tikzpicture}[anchorbase]
    \draw[->] (0,0) to (0,1);
  \end{tikzpicture}
  \qquad \text{and} \qquad
  \begin{tikzpicture}[anchorbase]
    \draw[->] (0.6,1) to (0.6,0.5) to[out=down,in=down,looseness=2] (0.3,0.5) to[out=up,in=up,looseness=2] (0,0.5) to (0,0);
  \end{tikzpicture}
  \ =\
  \begin{tikzpicture}[anchorbase]
    \draw[<-] (0,0) to (0,1);
  \end{tikzpicture}
  \ .
\end{equation}

To give a concrete example of duality in a monoidal category, consider the category $\Vect_\kk$ of finite-dimensional $\kk$-vector spaces, where $\kk$ is a field (see \cref{centerVect}).  In this category, the unit object is $\kk$.  We claim that, if $V$ is any finite-dimensional $\kk$-vector space, the dual vector space $V^*$ is both right and left dual to $V$ in the sense defined above.  Indeed, fix a basis $B$ of $V$, and let $\{\delta_v : v \in B\}$ denote the dual basis of $V^*$.  Viewing $V$ and $\uparrow$ and $V^*$ as $\downarrow$, we define
\begin{align*}
  \begin{tikzpicture}[anchorbase]
    \draw[->] (-0.3,0) to[out=down,in=down,looseness=2] (0.3,0);
  \end{tikzpicture}
  \ \colon \kk &\to V^* \otimes V, \quad \alpha \mapsto \alpha \sum_{v \in B} \delta_v \otimes v,
  &
  \begin{tikzpicture}[anchorbase]
    \draw[->] (-0.3,0) to[out=up,in=up,looseness=2] (0.3,0);
  \end{tikzpicture}
  \ \colon V \otimes V^* &\to \kk, \quad \sum_{i=1}^n v_i \otimes f_i \mapsto \sum_{i=1}^n f_i(v_i),
  \\
  \begin{tikzpicture}[anchorbase]
    \draw[<-] (-0.3,0) to[out=down,in=down,looseness=2] (0.3,0);
  \end{tikzpicture}
  \ \colon \kk &\to V \otimes V^*,\quad \alpha \mapsto \alpha \sum_{v \in B} v \otimes \delta_v,
  &
  \begin{tikzpicture}[anchorbase]
    \draw[<-] (-0.3,0) to[out=up,in=up,looseness=2] (0.3,0);
  \end{tikzpicture}
  \ \colon V^* \otimes V &\to \kk,\quad \sum_{i=1}^n f_i \otimes v_i \mapsto \sum_{i=1}^n f_i(v_i).
\end{align*}

Let's check the right-hand relation in \cref{rzigzag}.  The left-hand side is the composition
\begin{gather*}
  V \cong V \otimes \kk \xrightarrow{1_V \otimes\, \rcup} V \otimes V^* \otimes V \xrightarrow{\rcap \otimes 1_V} \kk \otimes V \cong V,
  \\
  w \mapsto w \otimes 1 \mapsto \sum_{v \in B} w \otimes \delta_v \otimes v \mapsto \sum_{v \in B} \delta_v(w) \otimes v \mapsto \sum_{v \in V} \delta_v(w) v = w.
\end{gather*}
Thus, this composition is precisely the identity map $1_V$, and so the right-hand relation in \cref{rzigzag} is satisfied.  The verification of the left-hand equality in \cref{rzigzag} and both equalities in \cref{lzigzag} are analogous and are left as an exercise for the reader.

If $\uparrow$ and $\downarrow$ are both left and right dual to each other, then we may form closed diagrams of the form
\begin{equation} \label{lintrace}
  \begin{tikzpicture}[anchorbase]
    \draw[<-] (0.5,0) arc(0:360:0.5);
    \filldraw[draw=black,fill=white] (-0.3,0) arc(0:360:0.2);
    \node at (-0.5,0) {\tokenlabel{f}};
  \end{tikzpicture}
  \ ,\quad f \in \End(\uparrow).
\end{equation}
Such closed diagrams live in the center $\End(\one)$ of the category.  Let's consider such a diagram in the category $\Vect_\kk$, where we know from \cref{centerVect} that the center of the category is isomorphic to $\kk$.  If $f \in \End_\kk(V)$, then the diagram \cref{lintrace} is the composition
\begin{gather} \nonumber
  \kk
  \xrightarrow{\lcup} V \otimes V^*
  \xrightarrow{f \otimes 1_{V^*}}  V \otimes V^*
  \xrightarrow{\rcap} \kk,
  \\ \label{centertrcomkp}
  \alpha
  \mapsto \alpha \sum_{v \in B} v \otimes \delta_v
  \mapsto \alpha \sum_{v \in B} f(v) \otimes \delta_v
  \mapsto \alpha \sum_{v \in B} \delta_v(f(v))
  = \alpha \tr(f),
\end{gather}
where $\tr(f)$ is the usual trace of the linear map $f$.  Therefore, under the isomorphism \cref{Endk}, the diagram \cref{lintrace} corresponds to $\tr(f)$.

\subsection{Mates\label{mates}}

Suppose that an object $X$ in a strict monoidal category has a right dual $X^*$.  Since we will now need to consider multiple objects with duals, we will denote the identity endomorphisms of $X$ and $X^*$ by upward and downward strands labeled $X$:
\[
  1_X =
  \begin{tikzpicture}[>=To,baseline={(0,0.15)}]
    \draw[->] (0,0) node[anchor=north] {$X$} to (0,0.5);
  \end{tikzpicture}
  \qquad \text{and} \qquad
  1_{X^*} =
  \begin{tikzpicture}[>=To,baseline={(0,0.15)}]
    \draw[<-] (0,0) node[anchor=north] {$X$} to (0,0.5);
  \end{tikzpicture}
  \ .
\]
As explained in \cref{duality}, the fact that $X^*$ is right dual to $X$ means that we have morphisms
\begin{equation} \label{pivcc}
  \begin{tikzpicture}[>=To,baseline={(0,-0.25)}]
    \draw[->] (-0.3,0) node[anchor=south] {$X$} to[out=down,in=down,looseness=2] (0.3,0);
  \end{tikzpicture}
  \ \colon \one \to X^* \otimes X
  \qquad \text{and} \qquad
  \begin{tikzpicture}[>=To,baseline={(0,0)}]
    \draw[->] (-0.3,0) node[anchor=north] {$X$} to[out=up,in=up,looseness=2] (0.3,0);
  \end{tikzpicture}
  \ \colon X \otimes X^* \to \one
\end{equation}
such that
\begin{equation} \label{pivzz}
  \begin{tikzpicture}[>=To,baseline={(0,0.5)}]
    \draw[<-] (0.6,0) to (0.6,0.5) to[out=up,in=up,looseness=2] (0.3,0.5) to[out=down,in=down,looseness=2] (0,0.5) to (0,1) node[anchor=south] {$X$};
  \end{tikzpicture}
  \ =\
  \begin{tikzpicture}[>=To,baseline={(0,0.5)}]
    \draw[<-] (0,0) to (0,1) node[anchor=south] {$X$};
  \end{tikzpicture}
  \qquad \text{and} \qquad
  \begin{tikzpicture}[>=To,baseline={(0,0.5)}]
    \draw[<-] (0.6,1) node[anchor=south] {$X$} to (0.6,0.5) to[out=down,in=down,looseness=2] (0.3,0.5) to[out=up,in=up,looseness=2] (0,0.5) to (0,0);
  \end{tikzpicture}
  \ =\
  \begin{tikzpicture}[>=To,baseline={(0,0.5)}]
    \draw[->] (0,0) to (0,1) node[anchor=south] {$X$};
  \end{tikzpicture}
  \ .
\end{equation}
Here we again label the strands with $X$ to distinguish between the cups and caps for different objects.

Suppose $X$ and $Y$ have right duals $X^*$ and $Y^*$, respectively.  Then every
\[
  \begin{tikzpicture}[anchorbase]
    \draw[->] (0,0) node[anchor=north] {$X$} to (0,1) node[anchor=south] {$Y$};
    \filldraw[black,fill=white] (0,0.7) arc(90:450:0.2);
    \node at (0,0.5) {\tokenlabel{f}};
  \end{tikzpicture}
  \ \in \End(X,Y)
  \qquad \text{has \emph{right mate}} \qquad
  \begin{tikzpicture}[anchorbase]
    \draw[<-] (0.6,-0.6) node[anchor=north] {$Y$} to (0.6,0) to[out=up,in=up,looseness=2] (0,0) to[out=down,in=down,looseness=2] (-0.6,0) to (-0.6,0.6) node[anchor=south] {$X$};
    \filldraw[black,fill=white] (0,0.2) arc(90:450:0.2);
    \node at (0,0) {\tokenlabel{f}};
  \end{tikzpicture}
  \ \in \End(Y^*,X^*).
\]

Now suppose $\cC$ is a strict monoidal category in which every object has a right dual.  Consider the map $R \colon \cC \to \cC$ that sends every object to its right dual and every morphism to its right mate.  How does $R$ behave with respect to vertical composition?  Omitting object labels, we have
\begin{multline*}
  R
  \left(
    \begin{tikzpicture}[anchorbase]
      \draw[->] (0,0) to (0,1);
      \filldraw[black,fill=white] (0,0.7) arc(90:450:0.2);
      \node at (0,0.5) {\tokenlabel{f}};
    \end{tikzpicture}
  \right)
  \circ
  R
  \left(
    \begin{tikzpicture}[anchorbase]
      \draw[->] (0,0) to (0,1);
      \filldraw[black,fill=white] (0,0.7) arc(90:450:0.2);
      \node at (0,0.5) {\tokenlabel{g}};
    \end{tikzpicture}
  \right)
  =
  \left(
    \begin{tikzpicture}[anchorbase]
      \draw[<-] (0.6,-0.6) to (0.6,0) to[out=up,in=up,looseness=2] (0,0) to[out=down,in=down,looseness=2] (-0.6,0) to (-0.6,0.6);
      \filldraw[black,fill=white] (0,0.2) arc(90:450:0.2);
      \node at (0,0) {\tokenlabel{f}};
    \end{tikzpicture}
  \right)
  \circ
  \left(
    \begin{tikzpicture}[anchorbase]
      \draw[<-] (0.6,-0.6) to (0.6,0) to[out=up,in=up,looseness=2] (0,0) to[out=down,in=down,looseness=2] (-0.6,0) to (-0.6,0.6);
      \filldraw[black,fill=white] (0,0.2) arc(90:450:0.2);
      \node at (0,0) {\tokenlabel{g}};
    \end{tikzpicture}
  \right)
  =
  \begin{tikzpicture}[anchorbase]
    \draw[<-] (1.8,-1.8) to (1.8,-1.2) to[out=up,in=up,looseness=2] (1.2,-1.2) to[out=down,in=down,looseness=2] (0.6,-1.2) to (0.6,-0.6) to (0.6,0) to[out=up,in=up,looseness=2] (0,0) to[out=down,in=down,looseness=2] (-0.6,0) to (-0.6,0.6);
    \filldraw[black,fill=white] (0,0.2) arc(90:450:0.2);
    \node at (0,0) {\tokenlabel{f}};
    \filldraw[black,fill=white] (1.2,-1) arc(90:450:0.2);
    \node at (1.2,-1.2) {\tokenlabel{g}};
  \end{tikzpicture}
  \stackrel{\text{(interchange)}}{=}
  \begin{tikzpicture}[anchorbase]
    \draw[<-] (1.8,-1.8) to (1.8,0) to[out=up,in=up,looseness=2] (1.2,0) to (1.2,-0.6) to[out=down,in=down,looseness=2] (0.6,-0.6) to[out=up,in=up,looseness=2] (0,-0.6) to (0,-1.2) to[out=down,in=down,looseness=2] (-0.6,-1.2) to (-0.6,0.6);
    \filldraw[black,fill=white] (0,-1) arc(90:450:0.2);
    \node at (0,-1.2) {\tokenlabel{f}};
    \filldraw[black,fill=white] (1.2,0.2) arc(90:450:0.2);
    \node at (1.2,0) {\tokenlabel{g}};
  \end{tikzpicture}
  \\
  \stackrel{\cref{rzigzag}}{=}
  \begin{tikzpicture}[anchorbase]
    \draw[->] (-0.6,1) to (-0.6,-0.3) to[out=down,in=down,looseness=2] (0,-0.3) to (0,0.3) to[out=up,in=up,looseness=2] (0.6,0.3) to (0.6,-1);
    \filldraw[black,fill=white] (0,-0.1) arc(90:450:0.2);
    \node at (0,-0.3) {\tokenlabel{f}};
    \filldraw[black,fill=white] (0,0.5) arc(90:450:0.2);
    \node at (0,0.3) {\tokenlabel{g}};
  \end{tikzpicture}
  =
  \begin{tikzpicture}[anchorbase]
    \draw[->] (-0.6,1) to (-0.6,-0.3) to[out=down,in=down,looseness=2] (0,-0.3) to (0,0.3) to[out=up,in=up,looseness=2] (0.6,0.3) to (0.6,-1);
    \filldraw[black,fill=white] (0,0.3) arc(90:450:0.3);
    \node at (0,0) {\tokenlabel{g \circ f}};
  \end{tikzpicture}
  = R
  \left(
    \begin{tikzpicture}[anchorbase]
      \draw[->] (0,0) to (0,1.2);
      \filldraw[black,fill=white] (0,0.9) arc(90:450:0.3);
      \node at (0,0.6) {\tokenlabel{g \circ f}};
    \end{tikzpicture}
  \right).
\end{multline*}
It follows that $R$ is a \emph{contravariant functor}.  In particular, for every object $X$, the functor $R$ induces a monoid anti-automorphism $\End(X) \to \End(X^*)$.  (If $\cC$ is strict $\kk$-linear monoidal, then this is an algebra anti-automorphism.)

In a manner analogous to the above, if $X^*$ and $Y^*$ are \emph{left} dual to $X$ and $Y$, respectively, then every
\[
  \begin{tikzpicture}[anchorbase]
    \draw[->] (0,0) node[anchor=north] {$X$} to (0,1) node[anchor=south] {$Y$};
    \filldraw[black,fill=white] (0,0.7) arc(90:450:0.2);
    \node at (0,0.5) {\tokenlabel{f}};
  \end{tikzpicture}
  \ \in \End(X,Y)
  \qquad \text{has \emph{left mate}} \qquad
  \begin{tikzpicture}[anchorbase]
    \draw[<-] (-0.6,-0.6) node[anchor=north] {$Y^*$} to (-0.6,0) to[out=up,in=up,looseness=2] (0,0) to[out=down,in=down,looseness=2] (0.6,0) to (0.6,0.6) node[anchor=south] {$X^*$};
    \filldraw[black,fill=white] (0,0.2) arc(90:450:0.2);
    \node at (0,0) {\tokenlabel{f}};
  \end{tikzpicture}
  \ \in \End(Y^*,X^*).
\]
This gives another contravariant endofunctor of $\cC$.

\subsection{Pivotal categories}

Let $\cC$ be a strict monoidal category.  Suppose that all objects have right duals, and $(X^*)^* = X$ for every object $X$.  It then follows that $X^*$ is also \emph{left} dual to $X$ for every object $X$.  In particular, in the language of \cref{mates}, we define a left cup and cap labeled $X$ to be a right cup and cap, respectively, labeled $X^*$:
\[
  \begin{tikzpicture}[>=To,baseline={(0,-0.25)}]
    \draw[<-] (-0.3,0) to[out=down,in=down,looseness=2] (0.3,0) node[anchor=south] {$X$};
  \end{tikzpicture}
  \ :=\
  \begin{tikzpicture}[>=To,baseline={(0,-0.25)}]
    \draw[->] (-0.3,0) node[anchor=south] {$X^*$} to[out=down,in=down,looseness=2] (0.3,0);
  \end{tikzpicture}
  \qquad \text{and} \qquad
  \begin{tikzpicture}[>=To,baseline={(0,0)}]
    \draw[<-] (-0.3,0) to[out=up,in=up,looseness=2] (0.3,0) node[anchor=north] {$X$};
  \end{tikzpicture}
  \ :=\
  \begin{tikzpicture}[>=To,baseline={(0,0)}]
    \draw[->] (-0.3,0) node[anchor=north] {$X^*$} to[out=up,in=up,looseness=2] (0.3,0);
  \end{tikzpicture}
  \ .
\]
If we also add the requirement that the duality data (i.e.\ the cups and caps) be compatible with the tensor product and that right mates always equal left mates, we get the following definition.

\begin{defin}[Strict pivotal category]
  A strict monoidal category $\cC$ is a \emph{strict pivotal category} if every object $X$ has a right dual $X^*$  with (fixed) morphisms \cref{pivcc} satisfying \cref{pivzz} and the following three additional conditions:
  \begin{enumerate}
    \item For all objects $X$ and $Y$ in $\cC$,
      \[
        (X^*)^* = X,\quad
        (X \otimes Y)^* = Y^* \otimes X^*,\quad
        \one^* = \one.
      \]

    \item For all objects $X$ and $Y$ in $\cC$, we have
      \[
        \begin{tikzpicture}[anchorbase]
          \draw[->] (-0.45,0) node[anchor=south] {$X \otimes Y$} to[out=down,in=down,looseness=2] (0.45,0);
        \end{tikzpicture}
        \ =
        \begin{tikzpicture}[anchorbase]
          \draw[->] (-0.6,0) node[anchor=south] {$X$} to[out=down,in=down,looseness=2] (0.6,0);
          \draw[->] (-0.3,0) node[anchor=south] {$Y$} to[out=down,in=down,looseness=2] (0.3,0);
        \end{tikzpicture}
        \qquad \text{and} \qquad
        \begin{tikzpicture}[anchorbase]
          \draw[->] (-0.45,0) node[anchor=north] {$X \otimes Y$} to[out=up,in=up,looseness=2] (0.45,0);
        \end{tikzpicture}
        \ =
        \begin{tikzpicture}[anchorbase]
          \draw[->] (-0.6,0) node[anchor=north] {$X$} to[out=up,in=up,looseness=2] (0.6,0);
          \draw[->] (-0.3,0) node[anchor=north] {$Y$} to[out=up,in=up,looseness=2] (0.3,0);
        \end{tikzpicture}
        \ .
      \]

    \item For every morphism $f \colon X \to Y$ in $\cC$, its right and left mates are equal:
      \begin{equation} \label{eq:equalmates}
        \begin{tikzpicture}[anchorbase]
          \draw[<-] (0.6,-0.6) node[anchor=north] {$Y$} to (0.6,0) to[out=up,in=up,looseness=2] (0,0) to[out=down,in=down,looseness=2] (-0.6,0) to (-0.6,0.6) node[anchor=south] {$X$};
          \filldraw[black,fill=white] (0,0.2) arc(90:450:0.2);
          \node at (0,0) {\tokenlabel{f}};
        \end{tikzpicture}
        \ =\
        \begin{tikzpicture}[anchorbase]
          \draw[<-] (-0.6,-0.6) node[anchor=north] {$Y$} to (-0.6,0) to[out=up,in=up,looseness=2] (0,0) to[out=down,in=down,looseness=2] (0.6,0) to (0.6,0.6) node[anchor=south] {$X$};
          \filldraw[black,fill=white] (0,0.2) arc(90:450:0.2);
          \node at (0,0) {\tokenlabel{f}};
        \end{tikzpicture}
        \ .
      \end{equation}
  \end{enumerate}
  It is important to note that a strict pivotal structure is extra data on the category $\cC$, namely the morphisms \cref{pivcc}.
\end{defin}

If $\cC$ is strict pivotal, and
\[
  \begin{tikzpicture}[anchorbase]
    \draw[->] (0,0) node[anchor=north] {$X$} to (0,1) node[anchor=south] {$Y$};
    \filldraw[black,fill=white] (0,0.7) arc(90:450:0.2);
    \node at (0,0.5) {\tokenlabel{f}};
  \end{tikzpicture}
  \in \Hom(X,Y)\ ,
\]
then we typically \emph{define} the corresponding coupon on a downward strand to be the right (equivalently, left) mate:
\[
  \begin{tikzpicture}[anchorbase]
    \draw[<-] (0,0) node[anchor=north] {$Y$} to (0,1) node[anchor=south] {$X$};
    \filldraw[black,fill=white] (0,0.7) arc(90:450:0.2);
    \node at (0,0.5) {\tokenlabel{f}};
  \end{tikzpicture}
  :=
  \begin{tikzpicture}[anchorbase]
    \draw[<-] (0.6,-0.6) node[anchor=north] {$Y$} to (0.6,0) to[out=up,in=up,looseness=2] (0,0) to[out=down,in=down,looseness=2] (-0.6,0) to (-0.6,0.6) node[anchor=south] {$X$};
    \filldraw[black,fill=white] (0,0.2) arc(90:450:0.2);
    \node at (0,0) {\tokenlabel{f}};
  \end{tikzpicture}
  \ =\
  \begin{tikzpicture}[anchorbase]
    \draw[<-] (-0.6,-0.6) node[anchor=north] {$Y$} to (-0.6,0) to[out=up,in=up,looseness=2] (0,0) to[out=down,in=down,looseness=2] (0.6,0) to (0.6,0.6) node[anchor=south] {$X$};
    \filldraw[black,fill=white] (0,0.2) arc(90:450:0.2);
    \node at (0,0) {\tokenlabel{f}};
  \end{tikzpicture}
  \ .
\]

Suppose a strict monoidal category $\cC$ is defined in terms of generators and relations, and that each generating object $X$ has a right dual generating object $X^*$, with $(X^*)^* = X$.  Then in order to show that $\cC$ is pivotal, it suffices to show that the right and left mates of each generating morphism are equal.  The axioms of a strict pivotal category then uniquely determine the duality data for arbitrary objects, which are tensor products of the generating objects.

In a strict pivotal category, isotopic string diagrams represent the same morphism!  (See \cite[\S2.4]{TV17} for a detailed discussion.)  This allows us to use geometric intuition and topological arguments in the study of such categories.  In some places in the literature, the strict pivotal nature of a category is implicit in the definition.  More precisely, categories where morphisms consist of planar diagrams \emph{up to isotopy} are strict pivotal by definition.  This is the case, for example, for the Heisenberg categories defined in \cite{Kho14,CL12,RS17} and the categorified quantum group of \cite{KL10}.

One of the simplest examples of a strict pivotal category is the \emph{Temperley--Lieb category} $\mathcal{TL}(\delta)$, $\delta \in \kk$.  This is a strict $\kk$-linear monoidal category on one generating object $X$.  We make this object self-dual by adding generating morphisms
\[
  \begin{tikzpicture}[anchorbase]
    \draw (-0.3,0) to[out=down,in=down,looseness=2] (0.3,0);
  \end{tikzpicture}
  \ \colon \one \to\ X \otimes X,
  \qquad
  \begin{tikzpicture}[anchorbase]
    \draw (-0.3,0) to[out=up,in=up,looseness=2] (0.3,0);
  \end{tikzpicture}
  \ \colon X \otimes X \to \one,
\]
and relations
\[
  \begin{tikzpicture}[>=To,baseline={(0,0.5)}]
    \draw (0.6,0) to (0.6,0.5) to[out=up,in=up,looseness=2] (0.3,0.5) to[out=down,in=down,looseness=2] (0,0.5) to (0,1);
  \end{tikzpicture}
  \ =\
  \begin{tikzpicture}[>=To,baseline={(0,0.5)}]
    \draw (0,0) to (0,1);
  \end{tikzpicture}
  \ =\
  \begin{tikzpicture}[>=To,baseline={(0,0.5)}]
    \draw (0.6,1) to (0.6,0.5) to[out=down,in=down,looseness=2] (0.3,0.5) to[out=up,in=up,looseness=2] (0,0.5) to (0,0);
  \end{tikzpicture}
  \ .
\]
We also impose the relation
\[
  \begin{tikzpicture}[anchorbase]
    \draw (0,0) circle(0.3);
  \end{tikzpicture}
  \ = \delta.
\]
The endomorphism algebra $\End_{\mathcal{TL}(\delta)}(X^{\otimes n})$ is the \emph{Temperley--Lieb algebra} $TL_n(\delta)$.

\section{Categorification\label{sec:cat}}

\subsection{Additive categories}

A $\kk$-linear category is said to be \emph{additive} if it admits all finitary biproducts (including the empty biproduct, which is a \emph{zero object}).  For example, if $A$ is an associative $\kk$-algebra, then the category of left modules over $A$ is additive, with biproduct given by the direct sum $\oplus$ of modules.

Given a $\kk$-linear category $\cC$, we can enlarge it to an additive category by taking its \emph{additive envelope} $\Add(\cC)$.  The objects of $\Add(\cC)$ are formal finite direct sums
\[
  \bigoplus_{i=1}^n X_i
\]
of objects $X_i$ in $\cC$.  Morphisms
\[
  f \colon \bigoplus_{i=1}^n X_i \to \bigoplus_{j=1}^m Y_j
\]
are $m \times n$ matrices, where the $(j,i)$-entry is a morphism
\[
  f_{i,j} \colon X_i \to Y_j.
\]
Composition is given by the usual rules of matrix multiplication.

\begin{eg}
  Let $A$ be an associative $\kk$-algebra, and let $\cC$ be the $\kk$-linear category of free rank-one left $A$-modules.  Then $\Add(\cC)$ is equivalent to the category of free left $A$-modules.
\end{eg}

\subsection{The Grothendieck ring}

Suppose $\cC$ is an additive $\kk$-linear category.  Let $\Iso_\Z(C)$ denote the free abelian group generated by isomorphism classes of objects of $\cC$, and let $[X]_\cong$ denote the isomorphism class of an object $X$.  Let $J$ denote the subgroup generated by the elements
\[
  [X \oplus Y]_\cong - [X]_\cong - [Y]_\cong,
  \quad X,Y \text{ objects of } \cC.
\]
The \emph{split Grothendieck group} of $\cC$ is
\[
  K_0(\cC) := \Iso_\Z(\cC) / J.
\]
In general, the split Grothendieck group is simply an abelian group.  However, if $\cC$ is an additive $\kk$-linear \emph{monoidal} category, then $K_0(\cC)$ is a ring with multiplication given by
\[
  [X]_\cong \cdot [Y]_\cong = [X \otimes Y]_\cong
\]
for objects $X$ and $Y$ (with the multiplication extended to all of $K_0(\cC)$ by linearity).

The process of passing to the split Grothendieck ring is a form of \emph{decategorification}.  The process of \emph{categorification} is a one-sided inverse to this procedure.  Namely, to categorify a ring $R$ is to find monoidal category $\cC$ such that $K_0(\cC) \cong R$ as rings.

\begin{eg}[Categorification of the ring of integers] \label{Zcat}
  Suppose $\kk$ is a field and let $\Vect_\kk$ be the category of finite-dimensional $\kk$-vector spaces.  This is an additive $\kk$-linear category under the usual direct sum of vector spaces.  Up to isomorphism, every vector space is determined uniquely by its dimension.  Thus
  \[
    \Iso_\Z(\C) = \Span_\Z \{[\kk^n]_\cong : n \in \N\}.
  \]
  Now, for an $n$-dimensional vector space $V$, we have
  \[
    V \cong \kk^{\oplus n},
  \]
  and so $[V]_\cong = n [\kk]_\cong$ in $K_0(\Vect_\kk)$.  It follows that we have an isomorphism
  \begin{equation} \label{K0Vect}
    K_0(\Vect_\kk) \xrightarrow{\cong} \Z,\quad
    \sum_{i=1}^n a_i [V_i]_{\cong} \mapsto \sum_{i=1}^n a_i \dim V_i.
  \end{equation}
  Since, for finite-dimensional vector spaces $U$ and $V$, we have $\dim (U \otimes V) = (\dim U)(\dim V)$, the isomorphism \cref{K0Vect} is one of rings.  In other words, $\Vect_\kk$ is a categorification of the ring of integers.
\end{eg}

\subsection{The trace\label{sec:trace}}

There is another common method of decategorification, which we now explain.  (See also \cite[\S2.6]{TV17}.)  Suppose $\cC$ is a $\kk$-linear category.  The \emph{trace}, or \emph{zeroth Hochschild homology}, of $\cC$ is the $\kk$-module
\[
  \Tr(\cC) := \left( \bigoplus_X \End_\cC(X) \right) / \Span_\kk \{f \circ g - g \circ f\},
\]
where the sum is over all objects $X$ of $\cC$, and $f$ and $g$ run through all pairs of morphisms $f : X \to Y$ and $g : Y \to X$ in $\cC$.  We let $[f] \in Tr(\cC)$ denote the class of an endomorphism $f \in \End_\cC(X)$.

If the category $\cC$ is strict pivotal, we can think of the trace as consisting of diagrams on an annulus.  In particular, if
\[
  \begin{tikzpicture}[anchorbase]
    \draw (0,0) to (0,1);
    \filldraw[black,fill=white] (0,0.7) arc(90:450:0.2);
    \node at (0,0.5) {\tokenlabel{f}};
  \end{tikzpicture}
\]
is an endomorphism in $\cC$, then we picture $[f]$ as
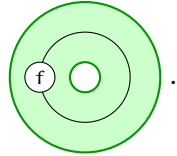
\[
  \begin{tikzpicture}[anchorbase]
    \filldraw[thick,draw=green!60!black,fill=green!20!white] (1,0) arc(0:360:1);
    \filldraw[thick,draw=green!60!black,fill=white] (0.2,0) arc(0:360:0.2);
    \draw (0.6,0) arc(0:360:0.6);
    \filldraw[draw=black,fill=white] (-0.4,0) arc(0:360:0.2);
    \node at (-0.6,0) {\tokenlabel{f}};
  \end{tikzpicture}
  \ .
\]
The fact that $[f \circ g] = [g \circ f]$ in $\Tr(\cC)$ then corresponds to the fact we can slide diagrams around the annulus:
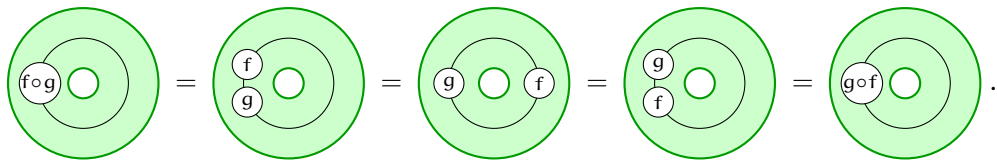
\[
  \begin{tikzpicture}[anchorbase]
    \filldraw[thick,draw=green!60!black,fill=green!20!white] (1,0) arc(0:360:1);
    \filldraw[thick,draw=green!60!black,fill=white] (0.2,0) arc(0:360:0.2);
    \draw (0.6,0) arc(0:360:0.6);
    \filldraw[draw=black,fill=white] (-0.3,0) arc(0:360:0.275);
    \node at (-0.6,0) {\tokenlabel{f \circ g}};
  \end{tikzpicture}
  \ =\
  \begin{tikzpicture}[anchorbase]
    \filldraw[thick,draw=green!60!black,fill=green!20!white] (1,0) arc(0:360:1);
    \filldraw[thick,draw=green!60!black,fill=white] (0.2,0) arc(0:360:0.2);
    \draw (0.6,0) arc(0:360:0.6);
    \filldraw[draw=black,fill=white] (-0.35,0.25) arc(0:360:0.2);
    \filldraw[draw=black,fill=white] (-0.35,-0.25) arc(0:360:0.2);
    \node at (-0.55,0.25) {\tokenlabel{f}};
    \node at (-0.55,-0.25) {\tokenlabel{g}};
  \end{tikzpicture}
  \ =\
  \begin{tikzpicture}[anchorbase]
    \filldraw[thick,draw=green!60!black,fill=green!20!white] (1,0) arc(0:360:1);
    \filldraw[thick,draw=green!60!black,fill=white] (0.2,0) arc(0:360:0.2);
    \draw (0.6,0) arc(0:360:0.6);
    \filldraw[draw=black,fill=white] (-0.4,0) arc(0:360:0.2);
    \node at (-0.6,0) {\tokenlabel{g}};
    \filldraw[draw=black,fill=white] (0.8,0) arc(0:360:0.2);
    \node at (0.6,0) {\tokenlabel{f}};
  \end{tikzpicture}
  \ =\
  \begin{tikzpicture}[anchorbase]
    \filldraw[thick,draw=green!60!black,fill=green!20!white] (1,0) arc(0:360:1);
    \filldraw[thick,draw=green!60!black,fill=white] (0.2,0) arc(0:360:0.2);
    \draw (0.6,0) arc(0:360:0.6);
    \filldraw[draw=black,fill=white] (-0.35,0.25) arc(0:360:0.2);
    \filldraw[draw=black,fill=white] (-0.35,-0.25) arc(0:360:0.2);
    \node at (-0.55,0.25) {\tokenlabel{g}};
    \node at (-0.55,-0.25) {\tokenlabel{f}};
  \end{tikzpicture}
  \ =\
  \begin{tikzpicture}[anchorbase]
    \filldraw[thick,draw=green!60!black,fill=green!20!white] (1,0) arc(0:360:1);
    \filldraw[thick,draw=green!60!black,fill=white] (0.2,0) arc(0:360:0.2);
    \draw (0.6,0) arc(0:360:0.6);
    \filldraw[draw=black,fill=white] (-0.3,0) arc(0:360:0.275);
    \node at (-0.6,0) {\tokenlabel{g \circ f}};
  \end{tikzpicture}
  \ .
\]

If $\cC$ is a $\kk$-linear monoidal category, then $\Tr(\cC)$ is a ring, with multiplication given by
\[
  [f] \cdot [g] = [f \otimes g].
\]
If $\cC$ is strict pivotal and we view elements of the trace as diagrams on the annulus, then this multiplication corresponds to nesting of annuli:
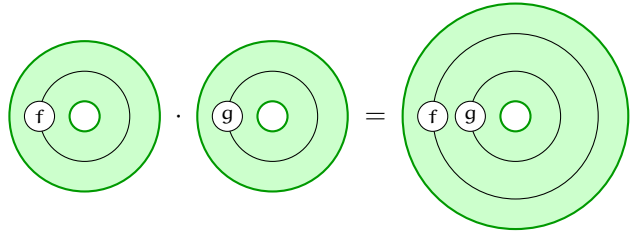
\[
  \begin{tikzpicture}[anchorbase]
    \filldraw[thick,draw=green!60!black,fill=green!20!white] (1,0) arc(0:360:1);
    \filldraw[thick,draw=green!60!black,fill=white] (0.2,0) arc(0:360:0.2);
    \draw (0.6,0) arc(0:360:0.6);
    \filldraw[draw=black,fill=white] (-0.4,0) arc(0:360:0.2);
    \node at (-0.6,0) {\tokenlabel{f}};
  \end{tikzpicture}
  \ \cdot\
  \begin{tikzpicture}[anchorbase]
    \filldraw[thick,draw=green!60!black,fill=green!20!white] (1,0) arc(0:360:1);
    \filldraw[thick,draw=green!60!black,fill=white] (0.2,0) arc(0:360:0.2);
    \draw (0.6,0) arc(0:360:0.6);
    \filldraw[draw=black,fill=white] (-0.4,0) arc(0:360:0.2);
    \node at (-0.6,0) {\tokenlabel{g}};
  \end{tikzpicture}
  \ =\
  \begin{tikzpicture}[anchorbase]
    \filldraw[thick,draw=green!60!black,fill=green!20!white] (1.5,0) arc(0:360:1.5);
    \filldraw[thick,draw=green!60!black,fill=white] (0.2,0) arc(0:360:0.2);
    \draw (0.6,0) arc(0:360:0.6);
    \filldraw[draw=black,fill=white] (-0.4,0) arc(0:360:0.2);
    \node at (-0.6,0) {\tokenlabel{g}};
    \draw (1.1,0) arc(0:360:1.1);
    \filldraw[draw=black,fill=white] (-0.9,0) arc(0:360:0.2);
    \node at (-1.1,0) {\tokenlabel{f}};
  \end{tikzpicture}
\]

We see from the above that the trace gives another method of \emph{decategorification}.  We thus have another corresponding notion of \emph{categorification}.  To categorify a $\kk$-algebra $R$ can mean to find a $\kk$-linear monoidal category $\cC$ such that $\Tr(\cC) \cong R$ as $\kk$-algebras.

To justify the use of the term \emph{trace}, consider the category $\Vect_\kk$ of finite-dimensional vector spaces over a field $\kk$.  Let $V$ be a $\kk$-vector space of finite dimension $n$.  Then we have an isomorphism
\[
  g \colon V \xrightarrow{\cong} \kk^n.
\]
For $1 \le a \le n$, define the inclusion and projection maps
\begin{align*}
  i_a &\colon \kk \to \kk^n,\quad \alpha \mapsto (\underbrace{0,\dotsc,0}_{a-1},\alpha,\underbrace{0,\dotsc,0}_{n-a}), \\
  p_a &\colon \kk^n \to \kk,\quad (\alpha_1,\dotsc,\alpha_n) \mapsto \alpha_a.
\end{align*}
Note that
\[
  p_b \circ i_a = \delta_{a,b} 1_\kk
  \qquad \text{and} \qquad
  \sum_{a=1}^n i_a \circ p_a = 1_{\kk^n}.
\]

Now suppose $f \colon V \to V$ is a linear map.  For $1 \le a, b \le n$, define
\[
  f_{a,b} = p_a \circ g \circ f \circ g^{-1} \circ i_b \colon \kk \to \kk.
\]
Then we have
\[
  f
  = g^{-1} \circ g \circ f \circ g^{-1} \circ g
  = \sum_{a,b=1}^n g^{-1} i_a f_{a,b} p_b g.
\]
Hence, in $\Tr(\Vect_\kk)$, we have
\begin{equation} \label{Tr=trace}
  [f]
  = \sum_{a,b=1}^n \left[ g^{-1} i_a f_{a,b} p_b g \right]
  = \sum_{a,b=1}^n \left[ f_{a,b} p_b g g^{-1} i_a \right]
  = \sum_{a,b=1}^n \left[ f_{a,b} p_b i_a \right]
  = \sum_{a=1}^n \left[ f_{a,a} \right].
\end{equation}
So the class of $[f]$ is equal to the sum of classes of endomorphisms of $\kk$ given by its diagonal components in some basis.  By \cref{Endk}, it follows that we have an isomorphism of rings
\begin{equation} \label{TrVectk}
  \Tr(\Vect_\kk) \cong \kk,\quad
  [f] \mapsto \tr(f).
\end{equation}
In particular, $\Vect_\kk$ is a trace categorification of the field $\kk$.

\subsection{Action of the trace on the center}

Suppose $\cC$ is a strict pivotal $\kk$-linear monoidal category.  We have seen that the trace $\Tr(\cC)$ can be thought of diagrams on the annulus, while the center $\End_\cC(\one)$ can be thought of as closed diagrams.  There is then a natural action of the trace on the center given by placing a closed diagram inside the inner boundary of the annulus, and then viewing the resulting diagram as a closed diagram.  In particular, if $f \in \End_\cC(X)$ and $z \in \End_\cC(\one)$ is a closed diagram, then the action of $[f] \in \Tr(\cC)$ on $z$ is
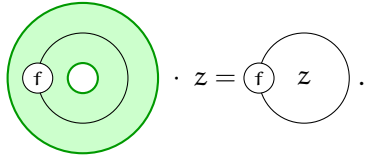
\[
  \begin{tikzpicture}[anchorbase]
    \filldraw[thick,draw=green!60!black,fill=green!20!white] (1,0) arc(0:360:1);
    \filldraw[thick,draw=green!60!black,fill=white] (0.2,0) arc(0:360:0.2);
    \draw (0.6,0) arc(0:360:0.6);
    \filldraw[draw=black,fill=white] (-0.4,0) arc(0:360:0.2);
    \node at (-0.6,0) {\tokenlabel{f}};
  \end{tikzpicture}
  \ \cdot\ z
  =
  \begin{tikzpicture}[anchorbase]
    \draw (0.6,0) arc(0:360:0.6);
    \filldraw[draw=black,fill=white] (-0.4,0) arc(0:360:0.2);
    \node at (-0.6,0) {\tokenlabel{f}};
    \node at (0,0) {$z$};
  \end{tikzpicture}
  \ .
\]
For example, \cref{lintrace} is the action of $[f]$ on the identity $1_\kk$ (the empty diagram) of the center $\End_\kk(\kk)$ of $\Vect_\kk$.  This explains the connection between \cref{centertrcomkp} and \cref{TrVectk}.

\subsection{The Chern character}

There is a nice relationship between the split Grothendieck group of a category and the trace of that category, as we now explain.  Suppose $\cC$ is an additive $\kk$-linear category.

\begin{lem}[{\cite[Lem.~3.1]{BGHL14}}] \label{tracesum}
  If $f \colon X \to X$ and $g \colon Y \to Y$ are morphisms in $\cC$, then
  \[
    [f \oplus g] = [f] + [g]
  \]
  in $\Tr(\cC)$.
\end{lem}

\begin{proof}
  We have $f \oplus g = (f \oplus 0) + (0 \oplus g) \colon X \oplus Y \to X \oplus Y$.  Thus
  \[
    [f \oplus g] = [f \oplus 0] + [0 \oplus g].
  \]
  Let
  \[
    i \colon X \to X \oplus Y
    \quad \text{and} \quad
    p \colon X \oplus Y \to X
  \]
  denote the obvious inclusion and projection.  Then
  \[
    [f \oplus 0] = [ifp] = [pif] = [f].
  \]
  Similarly, $[0 \oplus g] = [g]$.  This completes the proof.
\end{proof}

If $X$ and $Y$ are objects of $\cC$, then we have
\[
  1_{X \oplus Y} = 1_X \oplus 1_Y.
\]
Thus, by \cref{tracesum}, we have
\[
  [1_{X \oplus Y}]
  = [1_X \oplus 1_Y]
  = [1_X] + [1_Y].
\]
It follows that we have a well-defined map of abelian groups
\begin{equation} \label{Chern}
  h_\cC \colon K_0(\cC) \to \Tr(\cC),\quad
  h_\cC([X]_{\cong}) = [1_X].
\end{equation}
The map \cref{Chern} is called the \emph{Chern character map}.  If $\cC$ is an additive strict $\kk$-linear monoidal category, then $h_\cC$ is a homomorphism of rings.

In general, the Chern character map may not be injective and it may not be surjective.  See, for example, \cite[Examples~8--10]{BGHL14}.  However, there are some situations when it is an isomorphism.  A $\kk$-linear category $\cC$ is called \emph{semisimple} if it has finite direct sums, idempotents split (i.e.\ $\cC$ has subobjects), and there exist objects $X_i$, $i \in I$, such that $\Hom_\cC(X_i, X_j) = \delta_{i,j} \kk$ (such objects are called simple) and such that for any two objects $V$ and $W$ in $\cC$, the natural composition map
\[
  \bigoplus_{i \in I} \Hom_\cC(V,X_i) \otimes \Hom_\cC(X_i,W) \to \Hom_\cC(V,W)
\]
is an isomorphism.  If $\cC$ is an abelian category that is semisimple in the above sense and has a zero object, then $\cC$ is semisimple in the usual sense (i.e.\ all short exact sequences split).  If $\kk$ is algebraically closed, then the two notions are equivalent for abelian categories.  See \cite[p.~89]{Mug03} for details.

\begin{prop}
  If $\cC$ is a semisimple additive $\kk$-linear category, then the map
  \[
    h_\cC \otimes 1 \colon K_0(\cC) \otimes_\Z \kk \to \Tr(\cC)
  \]
  is an isomorphism.
\end{prop}

\begin{proof}
  Choose representatives $X_i$, $i \in I$, of the isomorphism classes of the simple objects in $\cC$.  As explained in \cite[p.~89]{Mug03}, every object in $\cC$ is a finite direct sum of the $X_i$.  Thus $K_0(\cC) \cong \bigoplus_{i \in I} \Z [X_i]_{\cong}$.  Suppose $Y$ is an object of $\cC$.  By assumption, there exist $\beta_{i,j} \in \Hom_\cC(Y,X_i)$ and $\alpha_{i,j} \in \Hom_\cC(X_i,Y)$ such that $1_Y = \sum_{i \in I} \sum_j \alpha_{i,j} \beta_{i,j}$.  Thus
  \[
    [1_Y] = \sum_{i \in I} \sum_j [\alpha_{i,j} \beta_{i,j}]
    = \sum_{i \in I} \sum_j [\beta_{i,j} \alpha_{i,j}]
    \in \sum_{i \in I} [\End_\cC(X_i)].
  \]
  Since $\Hom_\cC(X_i, X_j) = \delta_{i,j} \kk$, it follows that $\Tr(\cC) = \bigoplus_{i \in I} \kk [1_{X_i}]$.  It follows that $h_\cC \otimes 1$ is an isomorphism.
\end{proof}

There are other conditions on a category that imply the Chern character map is injective; see, for example, \cite[Prop.~2.5]{BHLW17}.

\subsection{Idempotent completions}

Suppose $R$ is a ring.  Throughout this subsection we will assume all $R$-modules are finitely-generated left modules.  The category of free $R$-modules is quite easy to work with, since all such modules are isomorphic to $R^n$ for some $n \in \N$.  However, we often want to work with the larger category consisting of \emph{projective} $R$-modules.  Fortunately, there is a natural relationship between these two categories.

Recall that an $R$-module is projective if and only if it is a direct summand of a free module.  In other words, an $R$-module $M$ is projective if and only if there exists another $R$-module $N$ and $n \in \N$ such that
\begin{equation} \label{projsum}
  M \oplus N \cong R^n
  \text{ as $R$-modules.}
\end{equation}
Now, given the isomorphism \cref{projsum}, let $p \colon R^n \twoheadrightarrow M$ denote the projection onto $M$, and let $i \colon M \hookrightarrow R^n$ denote the inclusion of $M$ into $R$.  Then
\[
  i \circ p \in \End_R(R^n)
\]
is an idempotent endomorphism of $R^n$ that corresponds to projection onto a submodule of $R^n$ isomorphic to $M$.

Conversely, if $e \in \End_R(R^n)$ is an idempotent (that is, $e^2=e$), then we have
\[
  R^n = e R^n \oplus (1-e) R^n.
\]
So the image $e R^n$ of $e$ is a projective $R$-module.

We see from the above that projective $R$-modules are precisely the images of idempotent morphisms of free $R$-modules.  So if we start with the category of free $R$-modules, we can enlarge this to the category of projective $R$-modules by adding objects corresponding to the images of idempotents.  This motivates the following definition.

\begin{defin}[Idempotent completion]
  The \emph{idempotent completion} (or \emph{Karoubi envelope}) of a category $\cC$ is the category $\Kar(\cC)$ whose objects are pairs of the form $(X,e)$, where $X$ is an object of $\cC$ and $e \colon X \to X$ is an idempotent in $\cC$, and whose morphisms are triples
  \[
    (e,f,e') \colon (A,e) \to (A',e'),
  \]
  where $f \colon A \to A'$ is a morphism of $\cC$ such that $f = e' \circ f \circ e$.
\end{defin}

One should think of the idempotent completion as a way of formally adding in images of idempotents, where $(A,e)$ is thought of as the image of the idempotent $e$.  The original category $\cC$ embeds fully and faithfully into $\Kar(\cC)$ by mapping an object $A$ of $\cC$ to $(A,1_A)$.  The idempotent completion of the category of free $R$-modules is equivalent to the category of projective $R$-modules.

\section{Heisenberg categories\label{sec:Heiscat}}

We conclude with some examples of categories, defined using the concepts introduced above, that are the focus of current research.

\subsection{Categorification of symmetric functions}

Recall the strict $\kk$-linear monoidal category $\cS$ from \cref{Scat}.  We assume in this subsection that $\kk$ is a field of characteristic zero.  The objects of $\cS$ are precisely $\uparrow^{\otimes n}$, $n=0,1,2,
\dotsc$, where $\uparrow^{\otimes 0} := \one$.  Note that all of our generating morphisms are endomorphisms, that is, their domain and codomain are equal.  Hence
\[
  \Hom_\cS(\uparrow^{\otimes n}, \uparrow^{\otimes m}) = 0
  \quad \text{for } m \ne n.
\]
In particular, $\uparrow^{\otimes n}$ is not isomorphic to $\uparrow^{\otimes m}$ for $m \ne n$.  It follows that
\[
  K_0(\Add(\cS)) \cong \Z[x],\quad
  [\uparrow]_\cong \mapsto x,
\]
is an isomorphism of rings.

Now let's consider the idempotent completion $\Kar(\cS)$.  To describe all the objects of $\Kar(\cS)$, we need to know all the idempotents of
\[
  \End_\cS(\uparrow^{\otimes n}) \cong \kk S_n
  \quad \text{for } n=0,1,2,3,\dotsc.
\]
Fortunately, the idempotents in the algebra $\kk S_n$ are well known.  For each partition $\lambda$ of $n$, we have the corresponding \emph{Young idempotent}
\[
  e_\lambda \in \kk S_n.
\]
For example, the Young idempotents for the partitions $(n)$ and $(1^n)$ are the complete symmetrizer and antisymmetrizer:
\[
  e_{(n)} = \frac{1}{n!} \sum_{\pi \in S_n} \pi,
  \qquad
  e_{(1^n)} = \frac{1}{n!} \sum_{\pi \in S_n} (-1)^{\ell(\pi)} \pi,
\]
where $\ell(n)$ is the length of the permutation $\pi \in S_n$.

It follows that the indecomposable objects in $\Kar(\cS)$ are, up to isomorphism,
\[
  (\uparrow^{\otimes n}, e_\lambda),
  \quad n=0,1,2,\dotsc,\ \lambda \vdash n.
\]
One can show that
\[
  K_0(\Kar(\Add(\cS))) \cong \Sym,
\]
the ring of symmetric functions.  The isomorphism is given explicitly by
\[
  \left[ (\uparrow^{\otimes n}, e_\lambda) \right]_\cong
  \mapsto s_\lambda,
\]
where $s_\lambda$ is the \emph{Schur function} corresponding to the partition $\lambda \vdash n$.

\subsection{Base category\label{base}}

Recall the strict $\kk$-linear monoidal category $\AH^\dg$ from \cref{sec:dAHA}.  This category has one generating object $\uparrow$, generating morhpisms
\[
  \begin{tikzpicture}[anchorbase]
    \draw[->] (-0.25,-0.25) to (0.25,0.25);
    \draw[->] (0.25,-0.25) to (-0.25,0.25);
  \end{tikzpicture}
  \ \colon \uparrow \otimes \uparrow \ \to\ \uparrow \otimes \uparrow
  \qquad \text{and} \qquad
  \begin{tikzpicture}[anchorbase]
    \draw[->] (0,0) to (0,0.6);
    \redcircle{(0,0.3)};
  \end{tikzpicture}
  \ \colon \uparrow\ \to\ \uparrow,
\]
and relations
\[
  \begin{tikzpicture}[anchorbase]
    \draw[->] (0.3,0) to[out=up,in=down] (-0.3,0.6) to[out=up,in=down] (0.3,1.2);
    \draw[->] (-0.3,0) to[out=up,in=down] (0.3,0.6) to[out=up,in=down] (-0.3,1.2);
  \end{tikzpicture}
  \ =\
  \begin{tikzpicture}[anchorbase]
    \draw[->] (-0.2,0) -- (-0.2,1.2);
    \draw[->] (0.2,0) -- (0.2,1.2);
  \end{tikzpicture}
  \ ,\qquad
  \begin{tikzpicture}[anchorbase]
    \draw[->] (0.4,0) -- (-0.4,1.2);
    \draw[->] (0,0) to[out=up, in=down] (-0.4,0.6) to[out=up,in=down] (0,1.2);
    \draw[->] (-0.4,0) -- (0.4,1.2);
  \end{tikzpicture}
  \ =\
  \begin{tikzpicture}[anchorbase]
    \draw[->] (0.4,0) -- (-0.4,1.2);
    \draw[->] (0,0) to[out=up, in=down] (0.4,0.6) to[out=up,in=down] (0,1.2);
    \draw[->] (-0.4,0) -- (0.4,1.2);
  \end{tikzpicture}
  \ , \qquad \text{and} \qquad
  \begin{tikzpicture}[anchorbase]
    \draw[->] (0,0) -- (0.6,0.6);
    \draw[->] (0.6,0) -- (0,0.6);
    \redcircle{(0.15,.45)};
  \end{tikzpicture}
  \ -\
  \begin{tikzpicture}[anchorbase]
    \draw[->] (0,0) -- (0.6,0.6);
    \draw[->] (0.6,0) -- (0,0.6);
    \redcircle{(.45,.15)};
  \end{tikzpicture}
  \ =\
  \begin{tikzpicture}[anchorbase]
    \draw[->] (0,0) -- (0,0.6);
    \draw[->] (0.3,0) -- (0.3,0.6);
  \end{tikzpicture}\ .
\]

Let's add another generating object $\downarrow$ that is right dual to $\uparrow$.  As noted in \cref{duality} this means that we have morphisms
\[
  \begin{tikzpicture}[anchorbase]
    \draw[->] (-0.3,0) to[out=down,in=down,looseness=2] (0.3,0);
  \end{tikzpicture}
  \ \colon \one \to \downarrow \otimes \uparrow
  \qquad \text{and} \qquad
  \begin{tikzpicture}[anchorbase]
    \draw[->] (-0.3,0) to[out=up,in=up,looseness=2] (0.3,0);
  \end{tikzpicture}
  \ \colon \uparrow \otimes \downarrow \to \one
\]
such that
\[
  \begin{tikzpicture}[anchorbase]
    \draw[<-] (0.6,0) to (0.6,0.5) to[out=up,in=up,looseness=2] (0.3,0.5) to[out=down,in=down,looseness=2] (0,0.5) to (0,1);
  \end{tikzpicture}
  \ =\
  \begin{tikzpicture}[anchorbase]
    \draw[<-] (0,0) to (0,1);
  \end{tikzpicture}
  \qquad \text{and} \qquad
  \begin{tikzpicture}[anchorbase]
    \draw[<-] (0.6,1) to (0.6,0.5) to[out=down,in=down,looseness=2] (0.3,0.5) to[out=up,in=up,looseness=2] (0,0.5) to (0,0);
  \end{tikzpicture}
  \ =\
  \begin{tikzpicture}[anchorbase]
    \draw[->] (0,0) to (0,1);
  \end{tikzpicture}
  \ .
\]

Let's define a right crossing by
\begin{equation} \label{rcross}
  \begin{tikzpicture}[anchorbase]
    \draw[->] (0,0) -- (0.6,0.6);
    \draw[<-] (0.6,0) -- (0,0.6);
  \end{tikzpicture}
  \ :=\
  \begin{tikzpicture}[anchorbase,scale=0.6]
    \draw[->] (0.3,0) -- (-0.3,1);
    \draw[->] (-0.75,1) -- (-0.75,0.5) .. controls (-0.75,0.2) and (-0.5,0) .. (0,0.5) .. controls (0.5,1) and (0.75,0.8) .. (0.75,0.5) -- (0.75,0);
  \end{tikzpicture}
\end{equation}

We will now define various categories by imposing one additional relation involving this right crossing.

\subsection{Affine oriented Brauer category}

Suppose that, in addition to the above generating objects, morphisms, and relations, we impose the additional relation that the right crossing \cref{rcross} is invertible.  This means that is has a two-sided inverse, which we will denote
\begin{equation} \label{lcross}
  \begin{tikzpicture}[anchorbase]
    \draw[<-] (0,0) -- (0.6,0.6);
    \draw[->] (0.6,0) -- (0,0.6);
  \end{tikzpicture}
  \ .
\end{equation}

The assertion that \cref{rcross,lcross} are two-sided inverses is precisely the statement that
\[
  \begin{tikzpicture}[anchorbase]
    \draw[->] (0.3,0) to[out=up,in=down] (-0.3,0.6) to[out=up,in=down] (0.3,1.2);
    \draw[<-] (-0.3,0) to[out=up,in=down] (0.3,0.6) to[out=up,in=down] (-0.3,1.2);
  \end{tikzpicture}
  \ =\
  \begin{tikzpicture}[anchorbase]
    \draw[<-] (-0.2,0) -- (-0.2,1.2);
    \draw[->] (0.2,0) -- (0.2,1.2);
  \end{tikzpicture}
  \qquad \text{and} \qquad
  \begin{tikzpicture}[anchorbase]
    \draw[<-] (0.3,0) to[out=up,in=down] (-0.3,0.6) to[out=up,in=down] (0.3,1.2);
    \draw[->] (-0.3,0) to[out=up,in=down] (0.3,0.6) to[out=up,in=down] (-0.3,1.2);
  \end{tikzpicture}
  \ =\
  \begin{tikzpicture}[anchorbase]
    \draw[->] (-0.2,0) -- (-0.2,1.2);
    \draw[<-] (0.2,0) -- (0.2,1.2);
  \end{tikzpicture}
  \ .
\]
Up to reflecting diagrams in a vertical axis, the resulting category $\mathcal{AOB}$ is the \emph{affine oriented Brauer category} defined in \cite{BCNR17}.  One can show that it is strict pivotal (see \cite[Th.~1.3]{Bru17}).  The left cups and caps are defined by
\[
  \begin{tikzpicture}[anchorbase]
    \draw[<-] (-0.3,0) to[out=down,in=down,looseness=2] (0.3,0);
  \end{tikzpicture}
  \ :=\
  \begin{tikzpicture}[anchorbase]
    \draw[<-] (-0.3,0.3) to[out=down,in=up] (0.3,-0.3) to[out=down,in=down,looseness=1.5] (-0.3,-0.3) to[out=up,in=down] (0.3,0.3);
  \end{tikzpicture}
  \qquad \text{and} \qquad
  \begin{tikzpicture}[anchorbase]
    \draw[<-] (-0.3,0) to[out=up,in=up,looseness=2] (0.3,0);
  \end{tikzpicture}
  \ :=\
  \begin{tikzpicture}[anchorbase]
    \draw[<-] (-0.3,-0.3) to[out=up,in=down] (0.3,0.3) to[out=up,in=up,looseness=1.5] (-0.3,0.3) to[out=down,in=up] (0.3,-0.3);
  \end{tikzpicture}
  \ .
\]
Omitting the dot generator yields the \emph{oriented Brauer category} $\mathcal{OB}$, which is the free symmetric monoidal category on a pair of dual objects.  It is obtained from the strict $\kk$-linear monoidal category $\cS$ of \cref{Scat} by adding a right dual object and inverting the right crossing as in \cref{base}.

The categories $\mathcal{OB}$ and $\mathcal{AOB}$ encode much of the representation theory of $\fgl_n(\kk)$.  Let $\fgl_n(\kk)$-mod denote the monoidal category of $\fgl_n(\kk)$-modules, and let $\End(\fgl_n(\kk)\md)$ be the monoidal category of endofunctors of $\fgl_n(\kk)$-mod.  Objects are functors $\fgl_n(\kk)\md \to \fgl_n(\kk)\md$, and morphisms are natural transformations.  For functors $F,G,F',G'$ and natural transformations $\eta \colon F \to F'$ and $\xi \colon G \to G'$, we define the tensor product by
\[
  F \otimes G := F \circ G,\qquad
  \eta \otimes \xi := \eta \xi \colon F \circ G \to F' \circ G'.
\]

Let $V$ be the natural $n$-dimensional representation of $\fgl_n(\kk)$, with dual representation $V^*$.  We have a monoidal functor
\begin{equation}\label{OB-functor}
  \mathcal{OB} \to \fgl_n(\kk)\md
\end{equation}
defined as follows.  On objects,
\[
  \uparrow\ \mapsto V,\qquad
  \downarrow\ \mapsto V^*,
\]
and, on morphisms,
\begin{align*}
  \begin{tikzpicture}[anchorbase]
    \draw[->] (-0.3,0) to[out=down,in=down,looseness=2] (0.3,0);
  \end{tikzpicture}
  \ &\mapsto
  \Big( \kk \to V^* \otimes V,\quad a \mapsto a \sum_{v \in B} \delta_v \otimes v \Big),
  \\
  \begin{tikzpicture}[anchorbase]
    \draw[->] (-0.3,0) to[out=up,in=up,looseness=2] (0.3,0);
  \end{tikzpicture}
  \ &\mapsto
  \Big( V \otimes V^* \to \kk,\quad v \otimes f \mapsto f(v) \Big),
  \\
  \begin{tikzpicture}[anchorbase]
    \draw[->] (-0.25,-0.25) to (0.25,0.25);
    \draw[->] (0.25,-0.25) to (-0.25,0.25);
  \end{tikzpicture}
  \ &\mapsto
  \Big( V \otimes V \to V \otimes V,\quad u \otimes v \mapsto v \otimes u \Big),
\end{align*}
where $B$ is a basis of $V$, and $\{\delta_v : v \in B\}$ is the dual basis.

Now, we have a natural monoidal functor
\[
  \fgl_n(\kk)\md \to \End(\fgl_n(\kk)\md),\quad M \mapsto M \otimes -,\quad f \mapsto f \otimes 1.
\]
Composition with \cref{OB-functor} yields a monoidal functor.
\[
  \mathcal{OB} \to \End(\fgl_n(\kk)\md).
\]
This can be extended to a monoidal functor
\[
  \mathcal{AOB} \to \End(\fgl_n(\kk)\md)
\]
by defining
\[
  \begin{tikzpicture}[anchorbase]
    \draw[->] (0,0) to (0,0.6);
    \redcircle{(0,0.3)};
  \end{tikzpicture}
  \ \mapsto
  \left( V \otimes - \to V \otimes -,\quad v \otimes w \mapsto \sum_{i,j=1}^n e_{i,j} v \otimes e_{j,i} w \right),
\]
where $e_{i,j}$ is the matrix with a $1$ in the $(i,j)$ position and a $0$ in all other positions.  This functor sends the center of $\mathcal{AOB}$ to $\End(\Id)$, which can be identified with the center of $U(\fgl_n(\kk))$.

\subsection{Heisenberg categories}

Fix $k \in \Z_{<0}$.  Let's return to the category of \cref{base}, but now impose the relation that the following matrix is an isomorphism in the additive envelope:
\begin{equation} \label{invrel}
  \begin{bmatrix}
    \begin{tikzpicture}[anchorbase]
      \draw[->] (0,0) -- (0.6,0.6);
      \draw[<-] (0.6,0) -- (0,0.6);
    \end{tikzpicture}
    &
    \begin{tikzpicture}[anchorbase]
      \draw[->] (-0.3,0.2) to (-0.3,0) to[out=down,in=down,looseness=2] (0.3,0) to (0.3,0.2);
    \end{tikzpicture}
    &
    \begin{tikzpicture}[anchorbase]
      \draw[->] (-0.3,0.2) to (-0.3,0) to[out=down,in=down,looseness=2] (0.3,0) to (0.3,0.2);
      \redcircle{(0.3,0)};
    \end{tikzpicture}
    &
    \cdots
    &
    \begin{tikzpicture}[anchorbase]
      \draw[->] (-0.3,0.2) to (-0.3,0) to[out=down,in=down,looseness=2] (0.3,0) to (0.3,0.2);
      \redcircle{(0.3,0)} node[anchor=west,color=black] {\dotlabel{-k-1}};
    \end{tikzpicture}
  \end{bmatrix}
  \colon (\uparrow \otimes \downarrow) \oplus \one^{\oplus (-k)} \to \downarrow \oplus \uparrow.
\end{equation}
We denote the resulting category by $\Heis_k$.

The inversion relation imposed above means that there is some $k \times 1$ matrix of morphisms in $\Heis_k$ that is a two-sided inverse to \cref{invrel}.  A thorough analysis of this category involves introducing notation for the entries of this inverse matrix and then deducing a simplified presentation of the category using the relations that arise from our two matrices being two-sided inverses to each other.  This is done in \cite{Bru17}, where it is shown that this category is strict pivotal and isomorphic to the \emph{Heisenberg category} introduced in \cite{MS17}.  In the case $k=-1$, this category was originally defined by Khovanov in \cite{Kho14}.  In \cite{MS17,Kho14}, the category was defined in terms of planar diagrams up to isotopy, and so was strict pivotal by definition.  The Heisenberg category encodes much of the representation theory of the symmetric group (when $k=-1$) and other degenerate cyclotomic Hecke algebras.  The affine oriented Brauer category can be thought of the $k=0$ version of the Heisenberg category.

Let us now explain why this category is called the \emph{Heisenberg} category.  The infinite-dimensional Heisenberg Lie algebra $\fh$ is the Lie algebra with generators
\[
  p_n^\pm,\ n=\Z_{>0,}
  \quad \text{and} \quad
  c,
\]
and relations
\[
  [p_n^+,p_m^+] = [p_n^-,p_m^-] = [c,p_n^\pm] = 0,\quad
  [p_n^+,p_m^-] = \delta_{n,m} n c,\qquad
  n,m \in \Z_{>0}.
\]
Since the element $c$ is central, it acts as a constant on any irreducible representation.  If we fix a \emph{central charge} $k \in \Z$, we can consider the associative algebra $U(\fh)/\langle c - k \rangle$, where $U(\fh)$ is the universal enveloping algebra of $\fh$.  Representations of $U(\fh)/\langle c - k \rangle$ are equivalent to representations of $\fh$ on which the central element $c$ acts as multiplication by $k$.  In particular, in $U(\fh)/\langle c - k \rangle$ we have
\begin{equation} \label{p1commute}
  p_1^+ p_1^- - k = p_1^- p_1^+.
\end{equation}

Now, the inversion relation \cref{invrel} implies that, in the Grothendieck group of the additive envelope of $\Heis_k$, we have
\begin{equation} \label{arrowcommute}
  [\uparrow]_\cong [\downarrow]_\cong + (-k) [\one]_\cong = [\downarrow]_\cong [\uparrow]_\cong.
\end{equation}
We see that \cref{p1commute,arrowcommute} are the same relation after replacing $p_1^+ \leftrightarrow [\uparrow]_\cong$, $p_1^- \leftrightarrow [\downarrow]_\cong$, $1 \leftrightarrow [\one]_\cong$.  In fact, if $\kk$ is a field of characteristic zero, then we have an isomorphism of algebras
\begin{equation} \label{Hisom}
  U(\fh)/\langle c - k \rangle \cong K_0(\Kar(\Add(\Heis_k))).
\end{equation}
Injectivity was proved in \cite[Th.~1]{Kho14} in the case $k=-1$ and in \cite[Th.~4.4]{MS17} in the general case $k < 0$.  It was conjectured in \cite[Conj.~1]{Kho14} and \cite[Conj.~4.5]{MS17} that \cref{Hisom} is an isomorphism.  This conjecture was recently proved in \cite[Th.~1.1]{BSW1}.  Earlier, an analogue of the conjecture was proved when one enlarges the Heisenberg category by adding in additional generating morphisms corresponding to elements of a graded Frobenius algebra with nontrivial grading.  See \cite[Th.~1]{CL12}, \cite[Th.~10.5]{RS17}, and \cite[Th.~1.5]{Sav18}.  The trace of the Heisenberg category has also been related to $W$-algebras in \cite{CLLS15}.

One can also define a \emph{quantum Heisenberg category} by replacing the symmetric group relations by the Hecke algebra relations \cref{braid-inv,braid,skein}.  See \cite{LS13,BSW2}.  The resulting category also categorifies the Heisenberg algebra.  Its trace has been related to elliptic Hall algebras in \cite{CLLSS16}.  The quantum analogue of the oriented Brauer category is the \emph{HOMFLY-PT skein category} introduced by Turaev in \cite[\S5.2]{Tur89}, where it was called the \emph{Hecke category}.  See also \cite{Bru17b,BSW2}.


\bibliographystyle{alphaurl}
\bibliography{biblist}

\end{document}